 \documentclass[final,1p,times]{elsarticleA}
\usepackage{amssymb}
\newcommand{\R}{{\Bbb R}}

\usepackage{amsmath}

\newtheorem{thm}{Theorem}
\newtheorem{lemma}[thm]{Lemma}
\newtheorem{corollary}[thm]{Corollary}
\newtheorem{proposition}[thm]{Proposition}

\newtheorem{definition}[thm]{Definition}
\newtheorem{remark}[thm]{Remark}
\newproof{proof}{Proof}

\begin{document}

\begin{frontmatter}



\title{Traveling waves for a model of the Belousov-Zhabotinsky reaction}

\author[a]{Elena Trofimchuk}
\author[b]{Manuel Pinto}
\author[d]{and Sergei Trofimchuk}
\address[a]{Department of Mathematics II,
National Technical University, Kyiv,  Ukraine
\\ {\rm E-mail: trofimch@imath.kiev.ua}}
\address[b]{Facultad de Ciencias, Universidad
de Chile,  San\-tia\-go, Chile  
\\ {\rm E-mail: pintoj@uchile.cl}}
\address[d]{Instituto de Matem\'atica y Fisica, Universidad de Talca, Casilla 747,
Talca, Chile \\ {\rm E-mail: trofimch@inst-mat.utalca.cl}}

\bigskip

\begin{abstract}
\noindent Following J.D. Murray, we consider a system of two  differential equations that models  traveling fronts  in the Noyes-Field theory of the Belousov-Zhabotinsky  (BZ) chemical reaction.  We are also interested in the situation when the system incorporates a delay $h\geq 0$. As we show, the BZ system has a dual character: it is monostable  when its key parameter $r \in (0,1]$ and it is bistable when $r >1$.  For $h=0, r\not=1$,  and for each admissible wave speed,  we prove the uniqueness  of monotone wavefronts.   Next, a concept of regular super-solutions is introduced as a main tool for  generating new comparison solutions for  the BZ system. This  allows  to improve all previously known upper  estimations for the minimal speed of propagation  in the BZ system, independently whether it is monostable, bistable, delayed  or not. Special attention is given to the critical case  $r=1$ which to some extent resembles to the Zeldovich equation. 
\end{abstract}

\end{frontmatter}

\newpage

\section{Introduction and main results}\vspace{0mm}
\noindent  One of useful  objects
associated with the famous Belousov-Zhabotinsky chemical reaction  is the following dimensionless
non-linear system \cite{Mur1,Mur2}
\begin{equation}\label{1}
\begin{array}{ll}
     u_t(t,x) = \Delta u(t,x)  + u(t,x)(1-u(t,x)-rv(t,x)),
    &    \\
     v_t(t,x) = \Delta v(t,x)  -b u(t,x)v(t,x), &
\end{array}%
\end{equation}
\noindent called the  Belousov-Zhabotinsky (BZ for short) reaction-diffusion system. The coefficients $r,b$ are
positive and $u,v$ correspond to the bromous acid and bromide ion
concentrations respectively. The front solution 
$(u,v) = (\phi, \theta)(\nu\cdot x + ct) $ of  system (\ref{1}) provides an
appropriate mathematical tool for the description of planar waves
propagating in a thin layer of reactant solution filled in a
Petri dish \cite{Mur2}. Due to the
chemical interpretation of  (\ref{1}), only non-negative
fronts are meaningful. Another requirement is the existence of
the limits $(\phi, \theta)(-\infty)=(0,a), \ (\phi, \theta)(+\infty)
= (1,0)$ with $a
>0$. The exact value of $a$ is not relevant: after rescaling $u,v$, we can take $a=1$.  By the experimental data \cite{Mur1,Mur2},  $r > 1$.  Nevertheless, 
almost all previous analytical studies of wavefronts (with two exceptions given in Propositions \ref{mu}, \ref{mu>1}) considered the case $r \in (0,1]$ which was proved to be of the monostable type. We observe that the standard definition   \cite{Volp} of monostability/bistability  needs an obvious modification in order to be applied to system (\ref{1})  which  has a continuum of non-negative equilibria.  The degeneracy of the equilibrium 
$(0,1)$ is a special feature of model (\ref{1})  complicating its analysis.  For example,  the
 recent Liang-Zhao general theory \cite{LZii} of spreading speeds for abstract monostable evolution systems can not be employed here despite the fact that system  (\ref{1}) is formally monostable and monotone for  $r\leq 1$.  This obliged us   in \cite{TPTv1}  to present  a complete proof of the existence of the minimal speed of front propagation in (\ref{1})  when $r\leq 1$. On the other hand,   we show here that, for each  $r>1$, the BZ system  possesses a unique wavefront solution, in full accordance with its formal bistability.  

Now, as it was argued in \cite{TPTv1}, a better theoretical prediction for propagation speeds in model (\ref{1}) can be also obtained by taking into account delayed effects during  the generation of  the bromous acid.  For simplicity, and in order to connect with various analytical investigations,  we will use here the following  delayed version of (\ref{1})  proposed by Wu and Zou in \cite{wz}: 
\begin{equation}\label{1de}
\begin{array}{ll}
     u_t(t,x) = \Delta u(t,x)  + u(t,x)(1-u(t,x)-rv(t-h,x)),
    &    \\
      v_t(t,x) = \Delta v(t,x)  -b u(t,x)v(t,x). &
\end{array}%
\end{equation}

During the last decades considerable efforts have been made in 
studying the wave propagation in (\ref{1}),  (\ref{1de}) . The attention was   
focused on the stability, numerical approximation \cite{MM, Mur2,Quin} and 
existence \cite{Ka1,Ka2,LL, Lv, ma,Mur1,Mur2,Troy,Volp,wz,yw}  of fronts.  After  linear changes, systems (\ref{1}), (\ref{1de}) acquire good monotonicity properties: they are quasi-monotone as partial differential equations  \cite{MS,Volp} and they  are  monotone in the sense of Wu and Zou \cite{wz}.  Hence,  the front existence may be handled by the standard comparison technique well established for several decades \cite{Volp,wz}. Thus the existence of fronts for the BZ system is not longer an issue,  in difference with the determination or satisfactory  approximation of  the minimal speed of  propagation in (\ref{1}), (\ref{1de}).  Precisely this problem   is our main concern here. It is quite noteworthy that a similar question (formulated as {\it linear  versus non-linear determinacy of the minimal speed}) for  a Lotka-Volterra  reaction-diffusion competition model 
has received a considerable attention during the last few years  \cite{GLi,wh,wh1}. Finally, our secondary concern is the uniqueness of wavefronts (cf. \cite{AGT}): since these have to be monotone, we prove their uniqueness in the non-delayed non-degenerate case (i.e. $r\not=1, h =0$) by applying the Berestycki-Nirenberg sliding solution argument \cite{BNa,chen}.  

\subsection{Some previously known  results}
\noindent For the sake of completeness, we state  the  most relevant  known  existence results for (\ref{1}), (\ref{1de}). First of them was proved in
\cite{Mur1,Mur2}, it gives a lower bound for the  admissible front speeds. Set
$$c_l:=\max\left\{2\Re\sqrt{1-r}, (\sqrt{r^2+2b/3}-r)/\sqrt{2b+4r}\right\}.
$$
\begin{proposition} \label{mu} Let $r,b >0$. If system (\ref{1}) has a positive componentwise monotone wavefront (or, shortly, monotone  wavefront)
connecting $(0,1)$ with $(1,0)$ then $c \geq c_l$. 
\end{proposition}
\vspace{-1.5mm}
It is easy to see that the estimation of Proposition
\ref{mu} has the form $c \geq c_l = 2\sqrt{1-r}$ for positive $r
\leq 11/12=0.917\dots$.
The next assertion summarizes the main existence results from
\cite{Ka1,Ka2,yw}. 
\vspace{-1.5mm}
\begin{proposition}\label{ka} System (\ref{1}) has a positive monotone  wavefront
$(u,v)(x,t) = (\phi, \theta)(\nu\cdot x + ct), $ \ $ |\nu|=1,$
connecting $(0,1)$ with $(1,0)$ for each velocity
$$ c \geq c_k= \left\{
\begin{array}{lll}
     2\sqrt{1-r},
    &  {\rm if} \ rb+r \leq 1 ;  \\ %\ b+r \not=1
      2\sqrt{b}, & {\rm if \ either} \ b+r > 1, b < 1, \ r \in (0,1] \ {\rm or} \ b=1, r <1;\\
 2, & {\rm if} \ b> 1, \ r \in (0,1].
\end{array}%
\right.
$$
\end{proposition}
\begin{proof} The first condition was proved in \cite[Theorem
3]{Ka2} under additional assumption $r+b >1$ when $b<1$. For
 $r,b >0$ satisfying $rb+r < 1$ it was also
announced without proof as Theorem 4.2 in \cite{yw}.  The
second and the third conditions were established in \cite[Theorem
2]{Ka2}. \hfill $\square$
\end{proof}
Model (\ref{1de}) was considered in \cite{bn,LL,Lv,ma,wz},  from where we have the following 
\vspace{0mm}
\begin{proposition}\label{kade} Assume that $r \in (0,1)$. Then system (\ref{1de}) has a positive monotone  front
$(u,v)(x,t)=$ $(\phi, \theta)(\nu\cdot x + ct), $ \ $ |\nu|=1,$
connecting $(0,1)$ with $(1,0)$ if either one of the following conditions holds: 
(I) $b >1$ and $c > \max\{b, 2 \sqrt{b}\}$;  (II) 
$c>2\sqrt{1-r}$ is such that
$
b\exp(-0.5ch(c-\sqrt{c^2-4(1-r)})) +r \leq 1.$  Finally, system (\ref{1de}) can not have 
wavefronts propagating  at the velocity $c < 2\sqrt{1-r}$. 
\end{proposition}
\vspace{-1mm}
Note  that in the non-delayed case  Proposition \ref{kade} is weaker than
Propositions \ref{mu}, \ref{ka}. 
\vspace{-2mm}
\begin{proof} See \cite[Theorem 3.1]{Lv} for condition $(I)$ and 
\cite[Theorem 3.2]{ma} for condition $(II)$.  The final conclusion is known 
from \cite{Mur2} (for $h =0$) and \cite{LL} (for $h \geq 0$).  \hfill $\square$  \end{proof}
By \cite[Section 8]{Mur1},   all  wavefronts to (\ref{1})  are monotone.  On the other hand,  the  delayed response  may imply the loss of wave's monotonicity \cite{GT}. Therefore  it is worthy to
emphasize that the inclusion of delay as in  (\ref{1de}) does not change the monotone shape of fronts, see \cite[Theorem 6]{TPTv1}:
\begin{proposition} \label{mude}   If, for some   $r,b >0$,  system (\ref{1de}) has a wavefront
$(u,v) = (\phi, \theta)$ $(\nu\cdot x + ct),$  $  \  \phi>0$, \
$ |\nu|=1,$ connecting $(0,1)$ with $(1,0)$,  then $\phi'(t),
-\theta'(t) > 0$ and $\theta(t), \phi(t) \in (0,1)$ for all $t \in
{\mathbf R}$.
\end{proposition}
\subsection{General remarks about our approach and some useful relations}\label{ss13}
The speed of front propagation in  (\ref{1}), (\ref{1de}) can be estimated  by means of  the  truncation method 
\cite{yw}, the shooting technique \cite{Ka1,Ka2}  and the upper and lower solutions \cite{Lv,ma,wz}.  Here, we use the latter approach   complementing it  by a useful idea  about how  to generate new  comparison solutions. 
The main working tool will be regular super-solutions  defined in Section \ref{upp}. 
 Theorem \ref{7} from the mentioned section  is instrumental for the proofs of existence: its application with different regular super-solutions yields Theorems \ref{main1}, \ref{main2}. The same  super-solutions are then used in the bistable case, see Theorems  \ref{mumi>1}, \ref{mumi>2}. Conceptually,  Theorem \ref{7} is very close to highly non-trivial Theorem 1(iv) from \cite{chen} (see also \cite{chenFG}).  The proofs of Theorem \ref{7} and the mentioned Chen and Guo   result   are, however, completely different. 

Asymptotic expansions of the eventual
fronts at infinity are another key ingredient of our approach. In combination with  a sliding solution argument they lead to 
\begin{thm}\label{mumi>1} Let $r \not=1, \ b >0$.  Then  for each fixed admissible wave speed $c$  the monotone wavefront  $(u,v)=
(\phi, \theta)(\nu\cdot x + ct), $  $ |\nu|=1,$  connecting equilibria $(0,1)$ and $(1,0)$ of system  (\ref{1}) is unique  (up to a translation).  
\end{thm}
\vspace{0mm}
We  also will need the following relations  between 
the components  of wavefront profile: 
\vspace{0mm}
\begin{thm}  \label{mda}  Consider  $\phi, \theta$ as in Proposition \ref{mude} and set $\psi(t): = 1- \theta (t)$. We have
\begin{enumerate}
\item[{\bf A.}]  Let $r \in (0,1), \ K \geq 1,   L \in (0,1]$  satisfy $K \geq  b/(1-r) \geq L$.  Then
\begin{equation}\label{oza}
L \phi(t-ch) < \psi(t) < K\phi(t), \quad t \in \mathbf{R}.
\end{equation}
If $b+r =1$, then  $\phi(t-ch) \leq  \psi(t)\leq \phi (t)$. Hence, if $h=0$ then $\phi \equiv \psi$. 
\item[{\bf B.}] 
Let $r\geq 1$. Then  $\psi(t) > \phi(t), \ t \in \R$. 
\item[{\bf C.}] 
Suppose that $r \in (0,1]$,  then  $\psi^2(t) < M\phi(t), \ t \in \R, \ M:=\max\{1,2b\}$.  
  \end{enumerate}
\end{thm} 
By part [A], the BZ system with  
$h=0, b+r =1$ essentially reduces to the KPP-Fisher equation \cite{GT, Mur1}.  Part [B] has a clear chemical interpretation:  the sum of the (normalized) concentrations of the bromous acid and bromide ion in the propagating wavefront
is strictly less than the concentration of the bromide ion far ahead of the wavefront. 
Part [C] connects  (\ref{1de})  with the delayed Zeldovich 
equation $
     u_t(t,x) = \Delta u(t,x)  + \beta u^2(t-h,x)(1-u(t,x)), \ \beta = \min\{b, 0.5\}.  
$
\noindent Actually this relation suggested the correct form of asymptotic expansions (\ref{a1}) below (see also 
\cite[Lemma 26 and Corollary 27]{TPTv1}). 

\subsection{Main results: monostable case}\label{mss13}
For the non-delayed BZ reaction  (\ref{1}) and $r \in (0,1)$, the existence of  the minimal speed of front propagation  $c_*(\Pi)$ was proved in \cite[p. 333]{Volp}.  
The speed $c_*(\Pi)$, however, is minimal only  for  the fronts  taking values in special  domains $\Pi$ called the balance polyhedrons.  Since the BZ system has a continuum of equilibria, none of these domains 
can cover the whole region  admissible for wavefronts, see \cite[Fig. 5.1, p. 334]{Volp}. 
The existence of the positive minimal speed independent on $\Pi$  was established  in  \cite[Theorem 7]{TPTv1}, by means of  regular super-solutions.  By Theorem 8 below, $c_* = 2 \sqrt{1-r}$  if $rb\exp(-2h(1-r))+r \leq 1$. However, due to 
Proposition \ref{mu},  it may happen that  $c_*$ is not 
linearly determined  (i.e.  $c_* > 2 \sqrt{1-r}$), cf. \cite{GLi,wh,wh1}.   Even for the non-delayed BZ system, the exact value of $c_*$ in 
the case $rb+r > 1$ is  unknown and represents an interesting open problem. The next theorems show that  the use of regular super-solutions  in the Wu and Zou approach \cite{wz}  yields important improvements of the estimations of $c_*$ even for the non-delayed  model. Set  $b':=be^{-c^2h/2}$ and let $c_\#= c_\#(r,b,h)$ be the unique positive root  \cite{TPTv1} of the equation
\vspace{-1mm} 
\begin{equation}\label{mins}
c=2\max\left\{\Re\sqrt{1-r},
\frac{\sqrt{b'}}{\sqrt{1+b'}}\right\} =  \left\{
\begin{array}{ll}
     2\sqrt{1-r},
    &  {\rm if} \ rb\exp(-2h(1-r))+r \leq 1;  \\
      2\sqrt{b'}/\sqrt{1+b'}, & {\rm if} \ rb\exp(-2h(1-r))+r \geq 1.
\end{array}%
\right. 
\end{equation}
\begin{thm} \label{main1} Let $r \in (0,1], \ c \geq c_\#$. Then 
system (\ref{1de}) has a positive monotone  front $(u,v)=
(\phi, \theta)(\nu\cdot x + ct), $  $ |\nu|=1,$ connecting $(0,1)$
with $(1,0)$ and such that    (i) if  $r=1, \ c > c_\#,$ then 
\begin{eqnarray} \label{a1}
\phi(t) &=&  \frac{2c^2/b}{t^2} -\frac{8c}{3b}(c^2(1+h+\frac1b)-4)\frac{\ln(-t)}{t^3} + O(\frac{1}{t^3}), \\
\nonumber \theta(t) &=&  1+\frac{2c}{t}- \frac{4}{3}(c^2(1+h+\frac1b)-4)\frac{\ln(-t)}{t^2} + O(\frac{1}{t^2}),  \ t \to -\infty;
\end{eqnarray}
(ii) if $r\in (0,1), \ c > 2\sqrt{1-r}$, then, for some $\varepsilon >0$ and $\lambda: = 0.5(c-\sqrt{c^2-4(1-r)})$, it holds
\vspace{-2mm} 
\begin{eqnarray}\label{a2}
\phi(t) =  e^{\lambda t} + O(e^{(\lambda+\varepsilon) t}) , \quad
 \theta(t) = 1 - \frac{be^{\lambda (t-ch)}}{1-r} + O(e^{(\lambda+\varepsilon) t}) ,  \  \  t \to -\infty. 
\end{eqnarray}
\vspace{-4mm}
\end{thm}
\begin{thm} \label{main2} Assume that $r \in (0,1], \ c \in [2\sqrt{1-r}, c_\#)$ and     
\begin{equation}\label{posmotrim}
f(c^2,r,b,h):= c^2(\frac{\omega_*}{8r}+ \frac h2)+ \ln \frac{c^2}{4br} - \frac{\omega_*}{2}\frac{1-r}{r} >0,
\end{equation}
where $\omega_*=8.21093\dots$ denotes the greatest positive root of the
equation $\omega = 4+ 2\ln \omega$.
Then system (\ref{1de}) has a positive monotone  front connecting $(0,1)$
with $(1,0)$.  Asymptotic formulas  (\ref{a2}) (when $c > 2\sqrt{1-r}$) and  (\ref{a1}) (when $r=1$) are fully applicable for this wavefront. 
 \end{thm}
Observe that inequality (\ref{posmotrim}) can be written as $c> c_\circ=c_\circ(r,b,h)$ where $c_\circ$ is the unique positive root of the equation $f(c^2,r,b,h)=0$ considered with fixed $r,b,h$. 

\subsection{Main results: bistable case}\label{bss13}
The next assertion can be considered as a dual to  Theorems \ref{main1},  \ref{main2}. Indeed, it essentially amounts to the non-existence  of bistable waves for $c > c_\#(r,b,h)$ and $c > c_\circ(r,b,h)$: 
\begin{thm}\label{mumi>2} Let $r >1, \ b >0$.  Then system  (\ref{1de}) has at most  one (a unique, if $h=0$) positive monotone wavefront $(u,v)=(\phi, \theta)$ $(\nu\cdot x + c_\star t),$  $  \  \phi>0, |\nu|=1,$ connecting $(0,1)$ with $(1,0)$.  The (unique) velocity of propagation  $c_\star$ satisfies the inequality $c_\star(r,b,h) \leq \min\{c_\#(r,b,h),c_\circ(r,b,h)\}$. 
In addition, $c_\star(r,b,h)$ is non-increasing in $h$.  
\end{thm} 
The wave existence problem for the bistable  BZ delayed system requires a different approach and it
is not considered here.  In the non-delayed case, the wavefront existence
was established by Kanel in \cite[Theorem 4]{Ka2}.  
In view of Theorem \ref{mumi>2}, Kanel's result  can be reformulated as 
\begin{proposition} \label{mu>1}Let $h=0, r >1$. Then  system (\ref{1}) has a positive monotone wavefront for the 
speed $c_\star$ such that 
$
c_K:=b/(2\sqrt{(r+b)\left[\min(1,b)(r+b)-0.5b\right]}) \leq c_\star < 2 \sqrt{\min(1,b)}.
$
\end{proposition}

\vspace{-5mm}

\begin{table}[h]
\caption{Analytical and numerical estimations of
$c_*, c_\star$}
\begin{center} \footnotesize
\begin{tabular}{|c|c|c|c|c|c|c|}
  \hline
  $(r;b)$ & Propositions \ref{ka}, \ref{mu>1} & Theorems \ref{main1}, \ref{mumi>2} & Theorems \ref{main2}, \ref{mumi>2}  & Numerical $c_*$ & Propositions \ref{mu},  \ref{mu>1} \\
  \hline
  $(0.5;5)$ & $c\geq 2$ & $c>1.82\dots$ & $c>1.62\dots $ & $c_*\approx1.46\dots $ & $c_* \geq c_l=1.414\dots $  \\
  $(0.5;10)$ & $c\geq 2$ & $c > 1.90\dots$ & $c > 1.71\dots$ & $c_*\approx 1.50\dots$ & $c_* \geq c_l=1.414\dots$  \\
  $(1;5)$ & $c\geq 2$ & $c > 1.82\dots$ & $c > 1.47\dots$ & $c_*\approx 1.13\dots$ & $c_* > c_l=0.289\dots$ \\
  $(5;0.5)$ & $c_\star \leq 1.41\dots$ & $c_\star \leq c_\#=1.15\dots $ & $c_\star \leq c_\circ=0.59\dots$ & $c_\star \approx 0.12\dots$ & $c_K= 0.067\dots, c_l=0.007\dots$ \\
  \hline
\end{tabular}\end{center} 
\label{tab1} 
\end{table}

\vspace{-3mm}

\noindent {\bf Example} In Table 1, for $h=0$, we compare results  of Theorems
\ref{main1}, \ref{main2}, \ref{mumi>2} with previously known  ones
(Propositions  \ref{ka}, \ref{mu>1}). Notation like
$c>1.82\dots$ means  that system (\ref{1}) has a positive
front for each velocity $c>1.82\dots$. Numerical
estimations of the minimal speed $c_*$ are taken from \cite[Table
3]{MM} and \cite[Table 1]{Quin}. Lower bounds for $c_*, c_\star$ are computed
from  Propositions \ref{mu}, \ref{mu>1}. 

\vspace{+2mm}

Finally, the organization of the paper is as follows.  Sections  \ref{OR2},  \ref{last}, \ref{4.2}, \ref{4.3a} and  \ref{5th} contain the proofs of Theorem  \ref{mda}, \ref{mumi>1}, \ref{main1}, \ref{main2} and \ref{mumi>2}, respectively. Asymptotic behavior of profiles at infinity is analyzed in  Section \ref{etomumi}.  Our main technical  result (Theorem \ref{7}) is proved in Section \ref{upp}.

\section{Proof of Theorem \ref{mda}} \label{OR2}
Let  $(u,v) = (\phi, \theta)(\nu\cdot x + ct)$ be a wavefront to  (\ref{1de}).  After introducing  
$\psi(t)= 1-  \theta(t-ch)$, we obtain the following boundary value problem for the determination of fronts in the BZ system:
 \begin{equation}\label{3ade}\left\{
\begin{array}{ll}
      \phi''(t) - c\phi'(t) + \phi(t) (1 - r- \phi(t)+ r\psi(t)) =0,
    &    \\
     \psi''(t) - c\psi'(t) +b \phi(t-ch)(1-\psi(t)) =0, & \\      \phi>0, \psi <1, \      \phi(-\infty)=\psi(-\infty)=0, \  \phi(+\infty)=\psi(+\infty)=1.  &\\ 
\end{array}%
\right.
\end{equation}
 [{\bf A}] 
Set $z(t):=K\phi(t)- \psi(t)$. 
It is easy to see that
$$
z''(t)-cz'(t)
+\left\{K(1-r)\phi(t) -b\phi(t-ch)+ K\phi(t)(r\psi(t)-\phi(t)) + b\psi(t)\phi(t-ch)\right\}=0.
$$
Since $z(-\infty)=0,\  z(+\infty)=K -1 \geq 0,$
the non-positivity of $z$ at some points  implies the existence of some $\tau$
such that $z(\tau)\leq 0, \ z'(\tau)= 0, \
z''(\tau)\geq 0$. But  $z(\tau)\leq 0$ implies
$
K\phi(\tau)\leq  \psi(\tau)
$
and therefore 
\begin{eqnarray*}
0&=& z''(\tau)
+\left\{K(1-r)\phi(\tau)+ K\phi(\tau)(r\psi(\tau)-\phi(\tau)) + b(\psi(\tau)-1)\phi(\tau-ch)\right\}\\
& > &
\left\{K(1-r)\phi(\tau)+ K\phi(\tau)(r\psi(\tau)-\phi(\tau)) + b(\psi(\tau)-1)\phi(\tau)\right\}\\
&\geq & 
\phi(\tau)\left\{[K(1-r) -b]+ K\phi(\tau)(rK-1+b)\right\}\geq 0,
\end{eqnarray*}
 a contradiction. The latter inequality holds obviously if  $rK-1+b \geq 0$. If 
  $rK-1+b < 0$, then 
$$
K(1-r) -b+ K\phi(\tau)(rK-1+b) >   K(1-r) -b+ K(rK-1+b)= 
(rK+b)(K-1) \geq 0.
$$
Next, set $z(t):=L\phi(t-ch)- \psi(t)$.  
We have  $z(-\infty)=0,\  z(+\infty)= L -1 \leq 0,$
so that the non-negativity of $z$ at some points would imply the existence of some
$\tau$ such that $z(\tau)\geq  0, \ z'(\tau)= 0, $ $
z''(\tau)\leq 0$. But then $
L\phi(\tau-ch)\geq  \psi(\tau)
$ 
and therefore 
\begin{eqnarray*}
0&=& z''(\tau) + (L(1-r)-b)\phi(\tau -ch) 
+L\phi(\tau-ch)(r\psi(\tau-ch)-\phi(\tau-ch))+ b \phi(\tau-ch)\psi(\tau)\\
& < &
(L(1-r)-b)\phi(\tau -ch) +
L\phi(\tau-ch)(r\psi(\tau)-\phi(\tau-ch))+ b \phi(\tau-ch)\psi(\tau)\\
&\leq & 
L_*:= (L(1-r)-b)\phi(\tau -ch) +
L\phi^2(\tau-ch)(rL-1+ b)\leq 0, 
\end{eqnarray*}
a contradiction. The latter inequality is obvious if $rL-1+b \leq 0$. If 
$rL-1+b > 0$ then 
$$L_* \leq \phi(\tau -ch) ((L(1-r)-b)+
L(rL-1+ b))=  \phi(\tau -ch) (rL+b)(L-1) \leq 0. 
$$

[{\bf B}]  Consider $z(t):= \psi(t) -\phi(t)$.  We have that $z(\pm\infty) =0$, 
\begin{equation}\label{tush}
 z''(t) -cz'(t) + b\phi(t-ch)(1-\psi(t)) - \phi(t)(1-r-\phi(t)+r\psi(t))=0, \quad t \in \R. 
\end{equation}
If $z(s) \leq 0$  at some $s$ then there exists $\tau$ such that $0\geq z(\tau) = \min_{t \in \R} z(t)$. We have 
 that $z''(\tau) \geq 0, $ $ \ z'(\tau) =0, \ \psi(\tau) \leq \phi(\tau),$ and 
 $$
- \phi(\tau)(1-r-\phi(\tau)+r\psi(\tau)) \geq - \phi(\tau)(1-r-\phi(\tau)+r\phi(\tau)) = \phi(\tau)(r-1)(1-\phi(\tau)) \geq 0, 
$$
contradicting to (\ref{tush}).

[{\bf C}]  Consider $z(t): = M\phi(t) - \psi^2(t)$. Since $M\geq 1$, we have $z(-\infty)=0,\  z(+\infty)=M -1 \geq 0$.  Thus
the non-positivity of $z$  implies that $z(\tau)\leq 0, \ z'(\tau)= 0, \
z''(\tau)\geq 0$ for some $\tau \in \R$. Hence,
$$
M\phi(\tau) \leq \psi^2(\tau), \quad M\phi'(\tau) = 2\psi(\tau)\psi'(\tau), \quad 
M\phi''(\tau) \geq 2\psi(\tau)\psi''(\tau)+ 2(\psi'(\tau))^2, 
$$
$$
\begin{array}{ll}
     0\geq  2\psi(\tau)\psi''(\tau)+ 2(\psi'(\tau))^2- 2c\psi(\tau)\psi'(\tau) + M\phi(\tau) (1-r +r\psi(\tau)- \phi(\tau)),
    &    \\
     0=2\psi(\tau)\psi''(\tau) - 2c\psi(\tau)\psi'(\tau) +2b\psi(\tau)\phi(\tau-ch)(1-\psi(\tau)),  & \\      
\end{array}%
$$
so that
\begin{eqnarray*}
0&\geq &  2(\psi'(\tau))^2 + M\phi(\tau) (1-r +r\psi(\tau)- \phi(\tau)) - 
2b\psi(\tau)\phi(\tau-ch)(1-\psi(\tau)) \\
& > &  M\phi(\tau) (1-r +r\psi(\tau)- \phi(\tau)) - 
2b\psi(\tau)\phi(\tau)(1-\psi(\tau))\\
& \geq &  \phi(\tau)\left\{M(1-r) +  \psi(\tau)(Mr-2b) +\psi^2(\tau)(2b-1)\right\}\geq 0,
\end{eqnarray*}
a contradiction. 
Here we observe that the polynomial $p(z):= M(1-r) +  z(Mr-2b) +z^2(2b-1), $ $z := \psi(\tau) \in (0,1), $ satisfies 
$
p(0)= M(1-r) \geq 0, \ p(1) = M-1 \geq 0, 
$
so that $p(\psi(\tau)) \geq 0$   if $2b -1 \leq 0$. If $2b-1 >0$
then we choose  $M:= 2b >1$ to obtain  
$$
 p(z) = 2b(1-r) - 2bz(1-r) +z^2(2b-1)= 2b(1-r)(1-z) +z^2(2b-1) >0. 
$$
\section{Asymptotics of wavefront profiles} \label{etomumi}
First, we observe that the derivatives $\phi', \psi'$ of  wavefront components are  bounded and uniformly continuous on ${\mathbf R}$ so that   $\phi'(\pm\infty)=\psi'(\pm\infty) =0$.  This fact is well known (cf. \cite[Section 2]{wz}) and its proof is omitted.   Incidentally, the relation $\psi'(\pm\infty) =0$ implies the positivity of each admissible speed (i.e. $c>0$):  it suffices to integrate  the second equation of  (\ref{3ade}) on ${\mathbf R}$ .  

Next,  assume that $r\in (0,1]$. 
Using Theorem \ref{mda}[B] if $r=1$ and 
integrating 
(\ref{3ade}) on $(-\infty,t],\ t \leq t_a$, we get, for sufficiently large negative $t_a$, 
$$
\phi'(t) < \phi'(t) + \int_{-\infty}^t\phi(s)(1-r +r\psi(s) - \phi(s))ds = c\phi(t),\  \psi'(t) + b\int_{-\infty}^t\phi(s-ch)(1-\psi(s))ds = c\psi(t),
$$
and therefore  $z(t):= \phi'(t)/\phi(t) <c, \ t \leq t_a$,  $\phi \in L_1({\mathbf R}_-)$.  Furthermore, $z$ satisfies the equation
\begin{equation} \label{prodok}
z' +z^2 -cz +(1-r) = f(t), \ {\rm where} \ f(t):= \phi(t)-r\psi(t).
\end{equation}
Let $\lambda = \lambda(c) \leq 
\mu= \mu(c)$ denote the roots of the characteristic equation
$x^2 -cx + (1-r)=0.
$ 
\begin{lemma} \label{pred} Let $(\phi, \psi)$ be a traveling front of (\ref{3ade}) and $r \in (0,1]$. Then 
(a) $c \geq 2\sqrt{1-r}$, (b) there exists finite limit $\lim \phi'(t)/\phi(t) \in \{\lambda, \mu\}$ as 
$t \to -\infty$.  \end{lemma}
\begin{proof} Recall that $f(t) \to 0$ as $t \to -\infty$. (a) Suppose that  $c < 2\sqrt{1-r}$.  Then, for some $t_b \leq t_a$, it holds  $f(t) -z^2+cz - (1-r) <0$ for $(t,z) \in (-\infty, t_b] \times [0,2c]$.  However, as a simple analysis of the direction field for equation (\ref{prodok}) shows,  this contradicts to the property $z(t) \in (0,c), \ t \leq t_a$. (b1)  Let $c = 2\sqrt{1-r}$ and take some small $\epsilon >0$. By analyzing the direction field again,  we can see that there exists $t_c$ such that 
$(t,z(t)) \in (-\infty,t_c] \times (-\epsilon +c/2, c/2+\epsilon)$ for $t \leq t_c$.  Hence, $z(t) \to c/2= \lambda=\mu$ as $t \to -\infty$. 
(b2) The situation when  $c > 2\sqrt{1-r}$ is similar to (b1). \hfill $\square$
\end{proof}
\begin{corollary} \label{posl} Let $r \in (0,1)$. Then there are  $t_1$, $m \in \{0,1\}$ and $\nu(c) \in \{\lambda(c), \mu (c)\}$ such  that   $(\psi(t+t_1), \phi(t+t_1),\phi'(t+t_1))= 
     (-t)^me^{\nu(c) t}(be^{-\nu(c)ch}/(1-r), 1, \nu(c))(1+o(1)), \  t \to -\infty.$ 
\end{corollary}
\begin{proof} By Lemma \ref{pred}, $\phi(t), \phi'(t)$ decay exponentially at $-\infty$. Then  $\psi(t)$ has the same property due to Theorem \ref{mda}[A].  Therefore we can apply Proposition 7.2 from \cite{MP} together with Theorem  \ref{mda}[A] to system (\ref{3ade}) 
in order to obtain the above asymptotic formulas for $\phi, \phi', \psi$.  Note that $m=1$ only when $c=2\sqrt{1-r}$.    \hfill $\square$
\end{proof}
\begin{lemma} \label{pred>1} Let $(\phi, \psi)$ be a wavefront for (\ref{3ade}) and $r >1$. Then,  
for some $ A > 0,\  t_2 \in {\mathbf R}$, and  small $\sigma >0$, it holds  
$
\phi(t+t_2) = e^{\mu(c)t} + O(e^{(2c- \sigma)t}), \ \psi(t+t_2) = Ae^{ct} + O(e^{(\mu{(c)}- \sigma)t}), \  t \to - \infty. 
$
\end{lemma}
\begin{proof} Integrating the first equation of (\ref{3ade})  from $-\infty$ to $t$, and using the inequality 
$1-r - \phi(t) +r \phi(t) <0$ for all large negative $t$ (say, for $t \leq T$ where, simplifying, we can take  $T =0$), we obtain that 
$\phi'(t) -c \phi(t) >0$ for $t \leq 0$. Thus $\phi(t) < \phi(0)e^{ct}, \ t \leq 0$.  Similarly, from the second equation of (\ref{3ade}),  we deduce $\psi(t) > \psi(0)e^{ct}, \ t \leq 0$.  The latter equation can be written as $\psi''(t) - c\psi'(t) = F(t)$, where $F(t):= b\phi(t-ch)(\psi(t) -1) = O(e^{ct}), $ $ 
\psi(t), \psi' (t) = o(1),$  $  t \to -\infty$.   But then \cite[Proposition 7.1]{MP} guarantees that $\psi(t), \psi'(t) = O (e^{(c-\sigma)t}), \ t \to -\infty,$ for each small $\sigma >0$.  Now, writing the first equation of (\ref{3ade}) as  $\phi''(t)-c\phi'(t) + (1-r)\phi(t) = G(t)$, where 
$G(t) = O(e^{(2c-\sigma)t}), \ t \to -\infty$, we find analogously that   $\phi(t) = Be^{\mu(c)t} + O(e^{(2c- \sigma)t}), \ t \to -\infty$,  where $\sigma >0, B \geq 0$.  To prove that $B >0$, it suffices to repeat the proof of Lemma \ref{pred} (note that $z(t)$ is bounded on  ${\mathbf R}_-$ because otherwise it blows up in a finite time).  Hence, 
$F(t)= O(e^{(\mu(c)t}), t \to -\infty$,   $\psi(t) > \psi(0)e^{ct}, \ t \leq 0,$ $\mu(c) >c$.  By  \cite[Proposition 7.1]{MP},   this yields the required asymptotic formula for $\psi$.  \hfill $\square$
\end{proof}
Next, we consider the case when  $t \to +\infty$. In order to linearize  system (\ref{3ade})
along the positive steady state $(1,1)$, we use the change of  variables $\phi(t) = 1 - \xi(t),$ $\psi(t)= 1-\theta(t-ch) $, which leads to
\begin{equation}\label{3adelin}\left\{
\begin{array}{ll}
      \xi''(t) - c\xi'(t) - \xi(t)(1-\xi(t)+r\theta(t-ch)) +r\theta(t-ch) =0,
    &    \\
     \theta''(t) - c\theta'(t) - b\theta (t) (1-\xi(t)) =0. &
\end{array}%
\right.
\end{equation}
The characteristic equation $(z^2-cz-1)(z^2-cz-b)=0 $ for this
system at the zero equilibrium has two positive ($\tilde\zeta_2, \zeta_2=0.5(c +\sqrt{c^2+4b})$) and two negative eigenvalues ($\tilde\zeta_1$  and  $\zeta_1 = 0.5(c -\sqrt{c^2+4b})$, respectively). 
\vspace{1mm}
\begin{lemma} \label{mainas} Let $r >0$. Then for some appropriate $A \geq 0$, $t_0, d, d_1$ and  small $\sigma>0$, we have  that
$(\phi(t+t_0),\phi'(t+t_0))= -Ae^{\tilde \zeta_1t}(1,\tilde \zeta_1)+$
$$\hspace{0mm}\nonumber
\left\{%
\begin{array}{lll}\nonumber (1- re^{\zeta_1(t-ch)}/(b-1), -r\zeta_1 e^{\zeta_1(t-ch)}/(b-1)) + O(e^{(\zeta_1-\sigma) t}), &  \
      b \not=1,
    \\ \nonumber 
    (1- r (t+d) e^{\zeta_1(t-ch)}/(c-2\zeta_1), -r\zeta_1 (t+d_1) e^{\zeta_1(t-ch)}/(c-2\zeta_1)) + O(e^{(\zeta_1-\sigma) t}),
    &  \  b=1 ,
\end{array}%
\right.\nonumber $$ 
$$(\psi(t+t_0),\psi'(t+t_0))=
(1-e^{\zeta_1 (t-ch)}, -\zeta_1e^{\zeta_1 (t-ch)})  +
O(e^{(\zeta_1-\sigma) t}), \ t \to +\infty.
$$
\end{lemma}
\begin{proof} 
Since $\theta(+\infty)=\xi(+\infty)=0$ and the linear system $y''(t) - cy'(t) -y(t)+r z(t-ch)=0,$ $z''(t) - cz'(t) -bz(t)=0$ possesses  an exponentially 
dichotomy on $\mathbf{R}_+$,  the perturbed system 
$$y''(t) - cy'(t) - y(t)(1-\xi(t)+r\theta(t-ch)) +rz(t-ch) =0, \ z''(t) - cz'(t) - bz(t) (1-\xi(t)) =0 $$
is also  exponentially  dichotomic on $\mathbf{R}_+$. 
As a consequence, we obtain that $
\theta(t), \theta'(t), \xi(t), \xi'(t) = O(e^{lt}),$ $t \to +\infty,$ for some negative
$l$.  Moreover, by applying  the Levinson asymptotic integration theorem \cite{MSPE} to the second equation of (\ref{3adelin}),  we find (cf. \cite[Lemma 19]{GT}) that, for some $t_0$,  
\begin{equation*}
(\theta(t+t_0),\theta'(t+t_0))=
(e^{\zeta_1 t}(1+o(1)), -\zeta_1e^{\zeta_1 t}(1+o(1))), \ t \to +\infty.
\end{equation*}
Then \cite[Proposition 7.2]{MP} applied to the 
second equation of (\ref{3adelin}) yields  the required estimation
\begin{equation}\label{1ax}
(\theta(t+t_0),\theta'(t+t_0))=
(e^{\zeta_1 t}, \zeta_1e^{\zeta_1 t}) +  O(e^{(\zeta_1-\sigma) t}) , \ t \to +\infty.
\end{equation}
Let  simplify (\ref{1ax}) by assuming $t_0=0$.  If $b\not=1$ then $\zeta_1 \not= \tilde \zeta_1$
and
$y = \xi (t)+ re^{\zeta_1(t-ch)}/(b-1)
=O(e^{lt})$ satisfies
\begin{equation}\label{ym}
y''(t)-cy'(t)-y(t)(1+m(t))= O(e^{(\zeta_1-\sigma)t}), \ t \to +\infty,
\end{equation}
where $m(t) = O(e^{lt})$. Applying again Proposition 7.2 from
\cite{MP}, we conclude that  if $\zeta_1 > \tilde\zeta_1$
(equivalently, $b \in (0,1)$) then $y(t), y'(t) =
O(e^{(\zeta_1-\sigma')t})$ with $\sigma' \in (0, \sigma)$ so that
$$(\xi (t), \xi'(t)) =\frac{re^{\zeta_1(t-ch)}}{1-b} (1, \zeta_1)
+O(e^{(\zeta_1-\sigma')t}), \ t \to +\infty.$$ When $\zeta_1
< \tilde \zeta_1$ (that is $b>1$), we find similarly that, for some 
$A >0$ and  $t \to +\infty$, 
$$0< \xi(t) =Ae^{\tilde\zeta_1t} + re^{\zeta_1(t-ch)}/(1-b)
+O(e^{(\zeta_1-\sigma')t}), \  \xi'(t) =A\tilde\zeta_1 e^{\tilde\zeta_1t} + r\zeta_1 e^{\zeta_1(t-ch)}/(1-b)
+O(e^{(\zeta_1-\sigma')t}).$$ 

\noindent Finally, if $b=1$ then  $\zeta_1 = \tilde\zeta_1$ and therefore 
$$y = \xi (t)+\frac{rte^{\zeta_1(t-ch)}}{2\zeta_1-c}
=O(e^{lt})$$ satisfies (\ref{ym}). As a consequence,  we obtain (once more invoking \cite[Proposition 7.2]{MP}) that, for some real $d, d_2$, it holds
$(\xi,\xi')(t) = r(t+ d, \zeta_1 t+ d_2)
e^{\zeta_1(t-ch)}/(c-2\zeta_1)+O(e^{(\zeta_1-\sigma')t}), \ t \to +\infty.$   \hfill $\square$
\end{proof}

\section{Proof of Theorem  \ref{mumi>1}} \label{last} \noindent The proof is based on the Berestycki-Nirenberg sliding solution argument.  Let $(\phi_1,\psi_1), \ (\phi_2,\psi_2)$ be two different  (modulo translation) traveling fronts 
of (\ref{3ade}) considered with $h=0$.  By Lemma \ref{mainas}, without restricting the generality, we may assume that $\psi_1$ and $\psi_2$ have the same first terms of their asymptotic expansions at $+\infty$. In addition, due to Lemma \ref{pred>1} (employed when $r>1$) and Corollary \ref{posl} (for $r<1$), we can index $\psi_j$ in such a way that either $\psi_1 (t) > \psi_2(t)$ on some infinite interval $(-\infty,T]$ or $\psi_1,\psi_2$ also have the same first asymptotic exponential terms at $-\infty$ (recall that  $r\not=1$). In each case, 
the closed set 
$
\mathcal{S}:= \{s : \psi_1(t+s) \geq \psi_2(t), \ t \in \R\} \not= \R
$ is non-empty and  contains  finite $s_*:=\inf \mathcal{S}$.   Similarly,  there exists 
the leftmost $t_*$ such that 
$
\phi_1(t+t_*) \geq \phi_2(t), \ t \in \R.
$ 

Let us show that actually $s_* =0$. Indeed, if $s_*>0$ then, 
due to the chosen asymptotic behavior of $\psi_j$ at $\pm \infty$, we find that, for each  $\varepsilon \in [0, s_*)$, it holds  $\psi_1(t+s_*-\varepsilon) > \psi_2(t)$ for all $t \in \R$ excepting $t$ from some compact interval.
This implies the existence of  finite $\bar t$ such that 
$
\delta(\bar t) =0, \ \delta''(\bar t) \geq 0, \ \delta(t):= \psi_1(t+s_*) - \psi_2(t) \geq 0.  
$  
If we suppose additionally that $s_*\geq t_*$ then $\phi_1(\bar t+s_*) - \phi_2(\bar t) > 0,\ t \in \R$.  
(Note that $\phi_1(\bar t+s_*) - \phi_2(\bar t) =0$ implies that $s_*=t_*$ and $\phi_1'(\bar t+s_*) - \phi_2'(\bar t) =0$. 
Since  also $\delta(\bar t) =0= \delta'(\bar t) =0$, the solution uniqueness theorem for (\ref{3ade}) assures that 
$(\phi_1,\psi_1)(t+s_*)\equiv  (\phi_2,\psi_2)(t)$).  But then  we get from (\ref{3ade}) the following contradiction:  
\begin{equation}\label{cru}
0= \delta''(\bar t) - c\delta'(\bar t) + b(\phi_1(\bar t+s_*)- \phi_2(\bar t))(1-\psi_2(\bar t)) >0. 
\end{equation}
Hence, we have to consider the case when $t_*>s_* >0$ and $\phi_1(\bar t+s_*)- \phi_2(\bar t) \leq 0$.  Note that 
$\delta(\pm\infty) =0, \delta (t) \geq 0,$ and therefore $\delta(t)$ has at least two local maxima  at some $t_j$: $t_1 < \bar t < t_2$. Since   $\delta''(t_j) \leq 0, \  \delta'( t_j) =0$, estimations similar to (\ref{cru}) shows that 
$\phi_1(t_j+s_*)- \phi_2(t_j) \geq 0, j =1,2$.  Next, set $S_a(t):=  \phi_1(t+s_*+a)- \phi_2(t)$. Functions $S_a(t)$ are  increasing in $a$ and  strictly positive on $[t_1,t_2]$ for all large $a >0$. On the other hand,  $S_0(t)$ has at least 
one zero on $(t_1,t_2)$. This means that for some $a_* \geq 0$ and $t_c \in (t_1,t_2)$ function $S_*(t):= S_{a_*}(t)$ reaches at $t_c$ its zero global minimum on $[t_1,t_2]$.  Therefore $S_*''(t_c) \geq 0, S_*'(t_c) = 0, S_*(t_c) =0,$ so that, due to (\ref{3ade}), 
\begin{equation}\label{cruz}
0= S_*''(t_c) - cS_*'(t_c) + r\phi_2(t_c)(\psi_1(t_c+s_*+a_*) -\psi_2(t_c)) \geq 0.   
\end{equation}
This shows that $a_*=0$ and that 
$$\psi_1'(\bar t+s_*)- \psi_2'(\bar t)=\psi_1(\bar t+s_*)- \psi_2(\bar t)=\phi_1'(\bar t+s_*)- \phi_2'(\bar t)=\phi_1(\bar t+s_*)- \phi_2(\bar t) = 0.$$ 
But then, by the uniqueness theorem for (\ref{3ade}), $(\phi_1,\psi_1)(t+s_*)\equiv (\phi_2,\psi_2)(t), \ t \in \R$ contradicting to our choice of  $(\phi_j,\psi_j)$.  Therefore we conclude that  $s_*=0$ and $\delta(t) >0,\ t \in \R$.  In the remainder of the proof we will analyze three possible mutual positions of $t_*$ and $0$.

\noindent \underline{Case A}: $t_* <0$.  Recall that $\psi_1(t) >  \psi_2(t)$, $\phi_1(t+t_*) \geq \phi_2(t)$.  Due to the coincidence of the  principal terms of asymptotic representations for $\psi_1,\psi_2$ at $+\infty$, we see that, for every small $\delta \in (0,|t_*|)$ the graphs of functions $\psi_1(t-\delta)$ and $\psi_2(t)$ have at least one intersection 
on some interval $[T, +\infty)$. In fact, we may assume that $\psi_1(T-\delta)> \psi_2(T)$ and 
$\psi_1(t-\delta)< \psi_2(t), t \in [T_1, +\infty),$ for some $T_1 >T$.  It is clear also that $\phi_1(t-\delta)> \phi_2(t)$ for all $t \in \R$.  Next, we consider the family of functions $\psi_1(t-\delta)+a$ and  the following non-empty and closed set 
$$
\frak{A}:= \{a\geq 0: \psi_1(t-\delta)+a \geq  \psi_2(t), \ t \in [T, +\infty) \}.
$$
Set  $a_* =\inf \frak{A}$,  it is evident that $a_* >0$ and that 
$w(t):= \psi_1(t-\delta)+a_* - \psi_2(t)$ has at least one zero  $t_p \in (T, +\infty)$, where, in addition,   $w'(t_p)=0,  w''(t_p) \geq 0$. But then, due to equations    (\ref{3ade}), 
$$
0 = w''(t_p) -c w'(t_p) + b\left(a_* \phi_2(t_p) + [\phi_1(t_p-\delta)-\phi_2(t_p)](1-  \psi_1(t_p-\delta))\right ) >0,   
$$
a contradiction proving that $t_* \geq 0$. In fact, we have established  a stronger result: for every $\delta >0$, the inequality $\phi_1(t-\delta)> \phi_2(t)$ does not hold on any infinite interval $[T,+\infty)$.  As a consequence, there exists a minimal $\rho \in [0,t_*]$ such that $\phi_1(t+\rho)\geq  \phi_2(t)$ for all  $t \in [T,+\infty)$.  That is, for every small $\delta >0$, equation  $\phi_1(t+\rho-\delta)=  \phi_2(t)$ has at least one root on $(T, +\infty)$
(otherwise, $\phi_1(t+\rho-\delta_j) <  \phi_2(t), \ t > T$, for  some $\delta_j \to  0$ and therefore 
$\phi_1(t+\rho) \leq \phi_2(t),\ t \geq T$, implying a contradiction:  $\phi_1(t+\rho) \equiv \phi_2(t), \ \psi_1(t+\rho) >  \psi_2(t), \  t \geq T$). 

\noindent \underline{Case B}: $t_* =0$, so that  $\psi_1(t) >\psi_2(t), \phi_1(t) \geq \phi_2(t), t \in \R$,  and, for each $\delta_k >0$, the inequalities $\phi_1(t-\delta_1)> \phi_2(t),  \ \psi_1(t-\delta_2)> \psi_2(t)$ do not hold on any  interval $[T,+\infty)$.  Now, it is easy to see that, in fact, $S_*(t): = \phi_1(t) - \phi_2(t) >0, t \in \R$.  Indeed, otherwise $S_*(t_c)=0$ for some $t_c$ and thus we get a contradiction as in (\ref{cruz}), where $s_*=a_*=0$ should be taken.  Hence, for a fixed $T$ and for  small $\delta >0$, each difference $\psi_1(t-\delta) - \psi_2(t), \phi_1(t-\delta) - \phi_2(t)$ has at least one zero on $[T,+\infty)$.  We can choose large $T$ and small $\delta >0$ in such a way  that 
\begin{equation} \label{star}
\phi_2(T) - 2r(1-\psi_1(T-\delta)) >0,  \quad \psi_1(T-\delta) > \psi_2(T), \quad \phi_1(T-\delta) > \phi_2(T)> 2/3.
\end{equation}
In the next stage of the proof, we apply the
sliding solution argument to the families $\epsilon + \psi_1(t-\delta)$ and $2\epsilon r + \phi_1(t-\delta)$.  It is clear 
that the sets 
$$
\mathcal{E}_1:= \{\epsilon \geq 0: \epsilon + \psi_1(t-\delta) \geq  \psi_2(t),\ t \in [T, +\infty) \},
$$
$$
\mathcal{E}_2:= \{\epsilon \geq 0: 2\epsilon r + \phi_1(t-\delta) \geq  \phi_2(t),\ t \in [T, +\infty) \}
$$
are closed and non-empty, and that $e_j = \inf \mathcal{E}_j$ are positive. Suppose first that $e_1\geq e_2$. The 
difference $\gamma(t):= e_1+ \psi_1(t-\delta) - \psi_2(t)$ reaches its global minimum at some point $t_m >T$ where $\gamma(t_m)= \gamma'(t_m) =0$ and $\gamma''(t_m) \geq 0$. We also have that 
$$
2e_1 r + \phi_1(t_m-\delta) \geq  2e_2 r + \phi_1(t_m-\delta) \geq \phi_2(t_m). 
$$
Therefore, using (\ref{3ade}) again, we find that 
$$
0 = \gamma''(t_m) - c \gamma'(t_m) + b\left[e_1\phi_2(t_m) + (\phi_1(t_m-\delta) - \phi_2(t_m))(1-\psi_1(t_m-\delta))\right]\geq 
$$
$$
be_1\left[\phi_2(t_m) - 2r(1-\psi_1(t_m-\delta))\right] >be_1\left[\phi_2(T) - 2r(1-\psi_1(T-\delta))\right]  > 0, 
$$
a contradiction.  So $e_1 < e_2$ and the
difference $\alpha(t):= 2e_2r+ \phi_1(t-\delta) - \phi_2(t)$ reaches its global minimum at some point $t_n >T$ where $\alpha(t_n)= \alpha'(t_n) =0$ and $\alpha''(t_n) \geq 0$. We also have that 
$$
e_2 + \psi_1(t_n-\delta) > e_1 + \psi_1(t_n-\delta) \geq \psi_2(t_n). 
$$
But then, after invoking (\ref{3ade}), we get  a contradiction:  
$$
0 = \alpha''(t_n) - c \alpha'(t_n) + r\phi_1(t_n-\delta)[2e_2+\psi_1(t_n-\delta)-\psi_2(t_n)] + 
$$
$$
2e_2r(r-1+\phi_1(t_n-\delta) +2e_2r -r \psi_2(t_n)) > 2e_2r (r(1-\psi_2(t_n)) -1+1.5\phi_1(t_n-\delta) +2e_2r) >0. 
$$
\noindent \underline{Case C}: $t_* >0$.  For a fixed large $T>0$, we consider $\phi_1(t+\rho)$ where $\rho \in [0,t_*]$ was defined in the last lines of subsection `Case A'. Then $\psi_1(t+\rho)> \psi_2(t), t \in \R,$ and, for each small $\delta >0$,  the equation  $\phi_1(t+\rho-\delta)=  \phi_2(t)$ has at least one root on $(T, +\infty)$.  From this point we can follow the proof given in Case B (beginning  from (\ref{star})). Actually, 
it will be literally the same proof if $\rho =0$. If $\rho >0$ we have to replace, starting from (\ref{star}), $\phi_1(t-\delta), \psi_1(t-\delta)$
with $\phi_1(t+\rho-\delta), \psi_1(t+\rho-\delta)$.  Note also that $e_1=0, e_2 >0$ if $\delta \in (0,\rho)$.

\section{Regular super-solutions and proof of Theorems  \ref{main1}, \ref{main2}} 
Assume that $r >0$ and $c^2>4(1-r)$. Recall that $\lambda = \lambda(c) <
\mu= \mu(c)$ denote the real roots of the characteristic equation
$
\chi(z,c):=z^2 -cz + (1-r)=0.
$
Fix some positive $\nu \in (\lambda, \mu)$. If $r \in(0,1)$ then we define  $k$ as the maximal positive integer such that $k\lambda \leq \nu$ and $(k+1)\lambda > \nu$. 
Obviously, if  $k >1$ then we have $\chi(j\lambda,c) <0$ for all $j=2,\dots,k$.
\subsection{Regular super-solutions and a preparatory theorem} \label{upp}
To prove the
existence of monostable  fronts, we will use   Wu and Zou  version \cite{wz} of the
upper and lower solutions method.   Below, we propose a trick which increases the effectiveness of this  approach for  the BZ system. We will show that it suffices
to find only two solutions (instead of four ones which 
must agree amongst themselves) of a system of differential
inequalities. 
 \begin{definition} \label{defo}{ Assume that continuous and piece-wise  $C^1-$smooth functions
$\psi_+, \ \phi_+$ are positive and have positive derivatives  in
some neighborhoods ${\mathcal O}_1,  \  {\mathcal O}_2$ of the sets
$(-\infty,t_1], \  (-\infty,t_2]$, respectively. We admit here that $(\psi'_+, \ \phi_+')$  has a finite set ${\mathcal{D}}=\{d_1< d_2 <...< d_M\}, $ $  \ d_M < \min\{t_1,t_2\}$,  of the discontinuity points  and one-sided derivatives of  $\psi_+, \ \phi_+$ satisfy $\psi'_+(d_j-) > \psi'_+(d_j+),$ $\phi'_+(d_j-) > \phi'_+(d_j+)$.  
Suppose also that $
\psi_+(-\infty) = \phi_+(-\infty) =0,$ $\psi_+(t_1) = \phi_+(t_2)
=1, $ and that $\psi_+, \  \phi_+$ are $C^2-$smooth in some
vicinities of $t_1, t_2$ and that
\vspace{2mm}
\begin{enumerate}
\item[{\bf D1.}] For a fixed positive $\nu \in (\lambda,\mu), \ m \in\{0,1\}$, and some positive constants $C_1,  \epsilon$, it holds 
$$
\hspace{-7mm}
\psi_+(t) = O(te^{\nu t}), \ \  (\phi_+(t),  \phi_+'(t), \phi_+''(t)) = C_1(-t)^me^{\nu t}(1,\nu,\nu^2)(1+o(1)), \ t \to -\infty.
$$
\item[{\bf D2.}] If $t \leq  \min\{t_1,t_2\}, \ t \not\in \mathcal{D}$, then 
\end{enumerate}
\begin{equation}
\label{3fs} \left\{\begin{array}{ll}
\Lambda_1(\phi_+,\psi_+)(t):= \phi_+''(t) - c\phi_+'(t) + \phi_+(t) (1 - r- \phi_+(t)+ r\psi_+(t)) <0,
      \\
\Lambda_2(\phi_+,\psi_+)(t):=     \psi_+''(t) - c\psi_+'(t) +b \phi_+(t-ch)(1-\psi_+(t)) < 0.
\end{array}%
\right.
\end{equation}
\begin{enumerate}
\item[{\bf D3.}] If $ t_1< t_2$ then
$ \phi_+''(t) - c\phi_+'(t) + \phi_+(t) (1 - \phi_+(t)) <  0,  \ t\in [t_1,t_2]$.

\item[{\bf D4.}] If $t_1> t_2$ then
$
 \psi_+''(t) - c\psi_+'(t) +b\min\{1, \phi_+(t-ch) \} (1-\psi_+(t)) < 0,$ $ t\in [t_2,t_1].
$
\end{enumerate}
\vspace{1mm}
We will call such $(\psi_+, \  \phi_+)$ a regular super-solution  for (\ref{3ade}).}  Observe that we may suppose that  $ \phi_+$ is defined, strictly increasing and smooth on $[t_2, +\infty)$, this fact is implicitly used in $\mathbf{D4}$.
\end{definition}
\begin{remark}\label{upro} Suppose that $\phi_+, \psi_+$ are increasing and  that 
inequalities (\ref{3fs}) hold for all $t \leq \max\{t_1,t_2\}$. 
Then conditions $\mathbf{D3, D4}$ are satisfied automatically. Indeed,  in case $\mathbf{D3}$, we have that $\Lambda_1(\phi_+,1) \leq  \Lambda_1(\phi_+,\psi_+) <0, \ t \in [t_1,t_2]$, while, in case $\mathbf{D4}$,  
$$
 \psi_+''(t) - c\psi_+'(t) +b\min\{1, \phi_+(t-ch) \} (1-\psi_+(t)) \leq  \Lambda_2(\phi_+,\psi_+) <0, \ t \in [t_2,t_1].
$$
\end{remark}
Note that  the  upper solutions for the BZ system proposed in earlier works 
(e.g. see  \cite{ma, wz})  have `correct' behavior at 
$-\infty$ and therefore do not satisfy condition $\mathbf{D1}$.  `Correct' here means `asymptotically similar to the true wavefront' (i.e. satisfying (\ref{a1}), (\ref{a2})). 
\vspace{-1mm}

\begin{thm} \label{7}   Suppose that for given parameters $b,c>2\sqrt{1-r}$,  $r \in (0,1]$,  $h\geq 0$, system (\ref{3ade})
has a regular super-solution $(\psi_+, \  \phi_+)$. Then there exists a monotone
wavefront for (\ref{1de})   moving at the velocity $c$  and satisfying (\ref{a1}), (\ref{a2}).
\end{thm}

\vspace{-1mm}

To prove Theorem \ref{7}, we will need several auxiliary statements. The
first of them  can be viewed as a variant of the Perron theorem for
piece-wise continuous solutions, cf. \cite{bn,GT}.

\vspace{-2mm}

\begin{lemma}\label{imp} Let $\psi: {\mathbf R} \to {\mathbf R}$ be a bounded classical solution of
the  impulsive equation
\begin{equation}\label{imi}
\psi'' + A\psi' + B\psi =f(t), \quad \Delta \psi|_{t_j} = \alpha_j,
\quad \Delta \psi'|_{t_j} = \beta_j,
\end{equation}
where $\{t_j\}$ is a finite increasing sequence, $f: {\mathbf R} \to {\mathbf R}$ is
bounded and continuous at every $t \not= t_j$ and $\Delta w|_{t_j} := w(t_j+)-w(t_j-)$. Assume that $\xi_1< 0<\xi_2$ are real roots of  $z^2 + Az + B =0$.
Then
\begin{eqnarray}\label{pfo} \hspace{10mm}\psi(t)&=&\frac{1}{\xi_1
- \xi_2} \left(\int^t_{-\infty} e^{\xi_1 (t-s)}f(s)ds +
\int_t^{+\infty}e^{\xi_2 (t-s)}f(s)ds\right) \\   \nonumber
 &+&\frac{1}{\xi_2 - \xi_1}\left[\sum_{t<t_j}e^{\xi_2(t-t_j)}(\xi_1\alpha_j
 -\beta_j)+\sum_{t>t_j}e^{\xi_1(t-t_j)}(\xi_2\alpha_j-\beta_j)\right], \
 \ t\not=t_j.
\end{eqnarray}
\end{lemma}
\begin{proof}
It is 
straightforward to check  that $\psi$ defined by (\ref{pfo})
verifies  equation (\ref{imi}).  \hfill $\square$ \end{proof}

\vspace{-2mm}

\begin{lemma} \label{fa} For  $r \in (0,1)$, set  $a_1:=1,\ b_1 := be^{-\lambda ch}/(1-r)$. There are functions
\begin{eqnarray*}
\phi_A(t)&:=& A(a_1e^{\lambda t} + a_2 e^{2\lambda t}+
\dots + a_k
e^{k\lambda t}), \\
\psi_A(t)&:=& A(b_1e^{\lambda t} + b_2 e^{2\lambda t}+
\dots +  b_k e^{k\lambda t}),
\end{eqnarray*}
and a polynomial $P(x,y)$ such that, for all $t \in {\mathbf R} $,
\begin{equation}\label{3} \hspace{5mm}\left\{
\begin{array}{ll}
    \phi_A''(t) - c\phi_A'(t) + \phi_A(t) (1 - r- \phi_A(t)+ r\psi_A(t)) =A^2P(e^{\lambda t},A)e^{\lambda (k+1)t},
    &    \\
    \psi_A''(t) - c\psi_A'(t) +b \phi_A(t-ch) = 0. &
\end{array}%
\right.
\end{equation}
\end{lemma}
\begin{proof}
Indeed, for a suitable  polynomial $P(x,y)$, we have that
\begin{eqnarray*} &&
\phi_A''(t) - c\phi_A'(t) + \phi_A(t) (1 - r- \phi_A(t)+
r\psi_A(t))  \\ 
&=&  A\sum_{j=1}^k(\chi(j\lambda,c)a_j -
\sum_{p+q=j}Aa_p(a_q-rb_q))e^{j\lambda t} + A^2P(e^{\lambda
t},A)e^{\lambda (k+1)t},\ {\rm and}
\end{eqnarray*}
\vspace{-8mm}
$$ \psi_A''(t) - c \psi_A'(t) +b
\phi_A(t-ch)=A\sum_{j=1}^k[\chi(j\lambda,c)b_j+(ba_je^{-j\lambda ch
}+(r-1)b_j)]e^{j\lambda t}.
$$
In order to obtain (\ref{3}), we define recursively  ($j=2,\dots,k$)
$$
a_1= 1, b_1 = \frac{be^{-\lambda ch}}{1-r}, \ a_j =
A\frac{\sum_{p+q=j}a_p(a_q-rb_q)}{\chi(j\lambda,c)}, \ b_j
=\frac{ba_je^{-j\lambda ch}}{1-r-\chi(j\lambda,c)}.  \qquad \square
$$  
 \end{proof}
\begin{remark} \label{pos} It is easy to see that, for  
some rational functions $a_j(b,r,\lambda,c,
\tau)$ and $b_j(b,r,\lambda,c, \tau)$, it holds
$$
a_j= A^{j-1}a_j(b,r,\lambda,c,e^{-\lambda ch}),  \ b_j =
A^{j-1}b_j(b,r,\lambda,c,e^{-\lambda ch}).
$$
Therefore
$A^{-1}(\phi_A(t), \psi_A(t))= (1,be^{-\lambda ch}(1-r)^{-1})e^{\lambda t}
+AO(e^{2\lambda t}), \ t \to - \infty. $  This implies that for
every $\tau_0 \in {\mathbf R} $ there exists $A_0 >0$ such that
$
\phi_A, \psi_A, \phi_A', \psi_A'
$ are positive 
for all $t \leq \tau_0, \ A \in (0,A_0]$. In addition, the derivatives of $\phi_A,\psi_A$ have the property $ \lim_{A\to
0+}(\phi^{(k)}_A(t),\psi^{(k)}_A(t))=(0,0), \ k =0,1,2,$ uniformly on $(-\infty,\tau_0]$. 
\end{remark}

\vspace{-2mm}

Next, in order to get an analog of $\phi_A,\psi_A$ when $r=1$, we consider the  functions
\begin{eqnarray*}
\phi_T(t)&:=&  \frac{2c^2/b}{t^2} + \frac{{\mathcal A}\ln(-t)}{t^3}  +\frac{T\ln^2(-t)}{t^4}, \\
\psi_Q(t)&:= &  -\frac{2c}{t}+ \frac{{\mathcal C}\ln(-t)}{t^2} + \frac{\mathcal F}{t^2} + \frac{Q\ln^2(-t)}{t^3}, \quad  t < -e,
\end{eqnarray*}
which coefficients ${\mathcal A,C,F}$ depend only on $c, b,h$ and are defined explicitly by  
\begin{eqnarray*}
&&{\mathcal A}:= -\frac{8c}{3b}(c^2(1+h+\frac1b)-4), \  {\mathcal C}:= \frac{4}{3}(c^2(1+h+\frac1b)-4),  \  ({\rm so \ that} \ b{\mathcal A}+2c{\mathcal C} =0); \\
&& {\mathcal F}:= \frac23(c^2(\frac1b-2-2h)-1)  \  (= \frac{2c^2}{b}-6+\frac{b}{2c}{\mathcal A} = 2(1-c^2-c^2h) + \frac12{\mathcal C}).
\end{eqnarray*}
Since 
$$
\frac{1}{(t-ch)^m} = \frac{1}{t^m} + \frac{mch}{t^{m+1}} + O\left(\frac{1}{t^{m+2}}\right ), \ \  
\frac{\ln^k(ch-t)}{(t-ch)^m} = \frac{\ln^k(-t)}{t^m}\left(1 + \frac{mch}{t}+ O\left(\frac{1}{t\ln(-t)}\right)\right) 
$$
at $t =-\infty,$ we find that 
\begin{eqnarray*}
\phi_T(t-ch)&:=&  \frac{2c^2/b}{t^2} + \frac{{\mathcal A}\ln(-t)}{t^3} + \frac{4c^3h/b}{t^3} +\frac{T\ln^2(-t)}{t^4}\left(1 + O(\frac1t)\right)+ O\left(\frac{\ln(-t)}{t^4}\right). 
\end{eqnarray*}
Then a straightforward computation shows that 
\begin{eqnarray*}
R_1(t):= \phi_T''(t) - c\phi_T'(t) + \phi_T(t) (\psi_Q(t)- \phi_T(t)) = r_{11}\frac{\ln^2(-t)}{t^5} +r_{12}\frac{\ln(-t)}{t^5} + \frac{O(1)}{t^5},
\end{eqnarray*}
where
$$
r_{11}:=  2cT+2c^2Q/b+{\mathcal A}{\mathcal C}, \quad r_{12}:= 12{\mathcal A}+{\mathcal AF}-4{\mathcal A}c^2/b- 2cT;
$$ 
\begin{eqnarray*}
\hspace{-3mm}\mbox{and} \quad R_2(t):= \psi_Q''(t) - c\psi_Q'(t) +b \phi_T(t-ch)(1-\psi_Q(t))  = 
r_{21}\frac{\ln^2(-t)}{t^4} +r_{22}\frac{\ln(-t)}{t^4} + \frac{O(1)}{t^4},
\end{eqnarray*}
with
$
r_{21}:=  bT+3cQ, \quad r_{22}:= 6{\mathcal C}(1-c^2) +3bch{\mathcal A}-2cQ.
$ 
\begin{lemma} \label{horo} There exist $T,Q$  and $\sigma=\sigma(T,Q,c,b,h)  >e$ such that $\phi_T(t) >0, $ $\psi_Q(t) >0$ and 
$R_1(t) <0, $ $ R_2(t) <0$ for all $t \leq   -\sigma$. 
\end{lemma}
\begin{proof} Take $T,Q$ such that $r_{11} >0$ and $r_{21} <0$.  Then it is easy to see that there is 
$\sigma= \sigma(T,Q,c,b,h) >e$ such that $t^2\phi_T(t) >0$, $ \ t^5R_1(t) > 0$  and $t \psi_Q(t) <0, $ $ t^4R_2(t) < 0$ for $t < -\sigma$.   \hfill $\square$
\end{proof}
\begin{lemma}  \label{ojala} For every $\epsilon >0$  there are $T_n,Q_n$ sufficiently large in absolute value and $\sigma_n=\sigma(T_n,Q_n,c,b,h)  <-e,$ such that the above defined functions $\phi_{T_n}, \psi_{Q_n}$  
\begin{enumerate}
\item[I.]
are positive and strictly increasing 
on the interval $(-\infty, \sigma_n]$;  
\item[II.] are strictly decreasing on the interval $(\sigma_n, \sigma_n+ch]$;
\item[III.] $\phi'_{T_n}(\sigma_n)=  \psi'_{Q_n}(\sigma_n) =0$ and  $\phi_{T_n}(\sigma_n)+  \psi_{Q_n}(\sigma_n) < \epsilon$;
\item[IV.] $R_1(t) >0 $ and  $R_2(t) >0$ for all $t \leq \sigma_n$. 
\end{enumerate}
\end{lemma}

\vspace{-4mm}

\begin{proof}  Take $\kappa_n \in [0.34, 0.98]\subset (1/3,1)$ and consider the sequences 
$T_n \to - \infty,$ $Q_n = -\kappa_n bT_n/c \to +\infty $.  It is easy to see that
$r_{11} <0$ and $r_{21} >0$ for all sufficiently large $n$.  

Now, it is clear that, for a given fixed interval $[z,-e]$, we have that $\phi_T(t) <0,$ $ t \in [z,-e]$
for all sufficiently large negative $T$. On the other hand, for each $T$, function $\phi_T$
is positive and strictly increasing on some interval $(-\infty, v]$. These simple observations show that to  every positive $\epsilon$ we can indicate $T_0<0$ such that, for each $T\leq T_0$, the functions 
$\phi_T, \phi'_T$ are positive on some maximal interval $(-\infty, \sigma_1(T))$ and $\phi_T(\sigma_1)<\epsilon/2, \ \phi'_T(\sigma_1)=0$.    The equation $\phi'_T(\sigma_1)=0$ can 
be written as $T= \Gamma_1(\sigma)$ with $\Gamma_1$ satisfying 
\begin{equation}\label{Ga}
\Gamma_1(\sigma_1):=-\frac{c^2}{b}\frac{\sigma_1^2}{\ln^2(-\sigma_1)}(1+o(1)),  \ \sigma_1 \to -\infty, 
\end{equation}
and strictly increasing on some maximal interval $(-\infty, \gamma_1(c,b,h)]$. Hence, 
we see that $\sigma_1 = \sigma_1(T) $ depends continuously on $T$ and monotonically converges to $-\infty$ as $T\to -\infty$.  

Furthermore, since the equation $
T= \Gamma_1(\sigma_1)$ has only one root $\sigma_1 \in (-\infty, \gamma_1(c,b,h)]$, we 
may suppose that   $\phi'_T(\sigma)<0$ for all $\sigma \in (\sigma_1, \sigma_1 +ch]$. 

Using (\ref{Ga}) and the monotonicity of $\Gamma_1$,  one can readily establish that 
$$
\sigma_1(T)= - \frac{\sqrt{-Tb}\ln(-T)}{2c}(1+o(1)), \quad T \to -\infty. 
$$
Similarly,  there is $Q_0>0$ such that, for each $Q\geq Q_0$, the functions 
$\psi_Q, \psi'_Q$ are positive on some maximal interval $(-\infty, \sigma_2(Q))$ and $\psi_Q(\sigma_2)<\epsilon/2, \ \psi'_Q(\sigma_2)=0$.    Equation $\psi'_Q(\sigma_2)=0$ can 
be written as $Q= \Gamma_2(\sigma_2)$ where 
$$
\Gamma_2(\sigma_2)= \frac{2c}{3}\frac{\sigma_2^2}{\ln^2(-\sigma_2)}(1+o(1)),  \ \sigma_2 \to -\infty, 
$$
strictly decreases on some maximal interval $(-\infty, \gamma_2(c,b,h)]$.  From this we deduce that $\sigma_2 = \sigma_2(Q) $ depends continuously on $Q$ and monotonically converges to $-\infty$ as $Q\to +\infty$.  Also  we may suppose that $\psi'_Q(\sigma)<0$ on $(\sigma_2, \sigma_2+ch]$. Next,   we have that 
$$
\sigma_2(Q_n)= -\sqrt{\frac{3Q_n}{8c}}(\ln Q_n)(1+o(1)) =  \sigma_1(T_n)\sqrt{\frac{3\kappa_n}{2}}(1+o(1)), \ n \to +\infty,
$$
and  since $3\kappa_n/2 \in [0.51,1.47]$, it is always possible to choose $\kappa_n$ in such a way that
 $\sigma_1(T_n)= \sigma_2(Q_n):= \sigma_n$ for all large $n$.  Obviously, $\kappa_n \to 2/3$.

Next, taking $Q=Q_n, \  T=T_n$, we find that for some  functions $\alpha_j, \ \beta_j$, uniformly on $n$ satisfying $\alpha_j(t) = o(1),$ $ \beta_j(t) = O(1),  \ t \to -\infty,$ it holds
$$
R_1(t)=  (2cT_n(1-\kappa_n+\alpha_1(t))+\beta_1(t))\frac{\ln^2(-t)}{t^5}  - \frac{\kappa_nbc^{-1}T^2_n}{t^7}\ln^4(-t)(1+\alpha_2(t)).  
$$
In this way, we prove  the existence of $\delta_1 <-e$ which does not depend on $n$  and such that $R_1(t) <0$  for all $t$ from some fixed interval $(-\infty, \delta_1]$.   Thus we may assume in the sequel that  $\sigma_n < \delta_1$. 

Analogously, we can use the representation
$$
R_2(t)= 
(bT_n(1-3\kappa_n)) +\alpha_3(t))\frac{\ln^2(-t)}{t^4} + \frac{\kappa_nb^2c^{-1}T^2_n}{t^7}\ln^4(-t)(1+\alpha_4(t)),
$$
to establish that $R_2$ is positive on some maximal interval $(-\infty, \delta_2)$, where $\delta_2=\delta_2(n)$ depends on $n$, $\lim \delta_2(n) = - \infty$ and $R_2(\delta_2(n))=0$. Analyzing the latter equation, we find that there is a  sequence $b_n \to b$ such that 
$$
\frac{3c}{2b_n}= \frac{T_n\ln^2(-\delta_2(n))}{\delta^3_2(n)}, \  {\rm and \ therefore \ } \  \delta_2(n) = \sqrt[3]{\frac{2bT_n}{27c}}\ln^{2/3}(-T_n)(1+o(1)). 
$$
Again, we have that $\sigma_n < \delta_2(n) $ for all large $n$, so that, without restricting the generality, we may suppose that both $R_1(t),R_2(t)$ are positive on $(-\infty, \sigma_n]$.   \hfill $\square$
\end{proof}
\begin{remark} \label{post}  For $r=1$ and small positive $A$, we will define 
$\phi_A, \psi_A$ by 
$$
\phi_A(t) = \phi_T(t-A^{-1}), \ \psi_A(t) = \psi_Q(t-A^{-1}),
$$
where sufficiently large $T,Q$ are chosen as in Lemma \ref{horo}.  
It is clear that for
every $\tau_0 \in {\mathbf R} $ there exists $A_0 >0$ such that
$
\phi_A(t) > 0, \psi_A(t) >0, \phi_A'(t) > 0, \psi_A'(t) >0
$
for all $t \leq \tau_0, \ A \in (0,A_0]$. In addition, $ \lim_{A\to
0+}(\phi^{(k)}_A(t),\psi^{(k)}_A(t))=(0,0), \ k =0,1,2,$ uniformly on $(-\infty,\tau_0]$. 
\end{remark}
Now we are in the position to prove Theorem \ref{7}.
\begin{proof} 
Consider the  functions
$$
\Phi_+(t,A)= \min\{1,\phi_A(t)+ \phi_+(t)\}, \ \Psi_+(t,A)=\min\{
1,\psi_A(t)+ \psi_+(t)\}, 
$$
where $\phi_A, \psi_A$ are defined in Remark \ref{post} if $r=1$ 
and  in Lemma \ref{fa} for $r \in (0,1)$. 
By Remarks \ref{pos}, \ref{post} and the implicit function theorem, there
exist smooth functions $\iota_1(A), \iota_2(A), $ such that $
\lim_{A\to 0+}\iota_1(A)=t_1, \lim_{A\to 0+}\iota_2(A)=t_2$ and
$$
\Phi_+(t,A), \Phi_+'(t,A) >0, \ t < \iota_2(A),\ {\rm with \ }
\Phi_+(t,A)=1,\ t \geq \iota_2(A),
$$
$$
\Psi_+(t,A), \Psi_+'(t,A) >0, \ t < \iota_1(A),\ {\rm with \ }
\Psi_+(t,A)=1,\ t \geq \iota_1(A),
$$
$$
\Phi'_+(\iota_2(A)-,A)> \Phi'_+(\iota_2(A)+,A)=0, \
\Psi'_+(\iota_1(A)-,A)> \Psi'_+(\iota_1(A)+,A)=0.
$$
We claim that for $t \not= \iota_1(A), \iota_2(A), d_1, \dots d_M$, and for
sufficiently small positive $A$, the functions
$\Phi_+(t):=\Phi_+(t,A), $ \ $\Psi_+(t):=\Phi_+(t,A)$ satisfy the
system
\begin{equation}\label{uspdd}\left\{
\begin{array}{lll}
\Phi_+''(t) - c\Phi_+'(t) + \Phi_+ (t)(1 - r- \Phi_+(t)+ r\Psi_+(t))
\leq 0, &    \\
\Psi_+''(t)- c\Psi_+'(t) +b \Phi_+(t-ch)(1-\Psi_+(t)) \leq 0, \
 & \\
 \Phi'_+(d_j-) > \Phi'_+(d_j+), \quad \Psi'_+(d_j-) > \Psi'_+(d_j+). 
 &
\end{array}%
\right.
\end{equation}
Since differential inequalities (\ref{uspdd}) hold trivially for  all $t \geq
\iota^*:=\max\{\iota_1(A), \iota_2(A)\}$ (when $\Phi_+(t)= \Psi_+(t)
=1$), it suffices to prove (\ref{uspdd})  for $t \in (-\infty,
\iota^*)$.  We will consider the following three cases.

\noindent \underline{Case I}. Let $t \leq \iota_*:=\min\{\iota_1(A),
\iota_2(A)\}$, then by Lemmas \ref{fa}, \ref{horo},   for all small $A>0$,
\begin{eqnarray}&&\nonumber
\Psi_+''(t) - c\Psi_+'(t) +b \Phi_+(t-ch)(1-\Psi_+(t))\leq  \psi_+''(t)
- c\psi_+'(t) +b \phi_+(t-ch)(1-\psi_+(t))   \\ \label{nonu}
&& \hspace{+2cm} 
- b (\phi_+(t-ch)\psi_A(t) + \phi_A(t-ch)\psi_+(t)) < 0,
\end{eqnarray}
due to assumption {\bf D2} of Definition \ref{defo} and the positivity of $\phi_+,\psi_A,\phi_A, \psi_+$. 
In a similar way  (but this time using  assumption {\bf D1})  we can evaluate $\Gamma$ defined by 
\begin{eqnarray}  \label{11}
\nonumber   \Gamma:&=&  \Phi_+''(t) - c \Phi_+'(t) +
\Phi_+(t) (1 - r - \Phi_+(t)+ r\Psi_+(t))  \\  \nonumber 
&=&\phi_+''(t) - c \phi_+'(t) + \phi_+(t) (1 -
r - \phi_+(t)+ r\psi_+(t)) + \phi_A''(t) - c \phi_A'(t) 
     \\  \nonumber  &+&
\phi_A(t) (1 - r - \phi_A(t)+
r\psi_A(t)) -2\phi_A(t)\phi_+(t)+ r\phi_A(t)\psi_+(t)+ r\phi_+(t)\psi_A(t). 
\end{eqnarray}
If $r \in (0,1)$ then we obtain 
\begin{eqnarray*}
 \Gamma=
 AO(e^{(\lambda+\nu) t})+
A^2O(e^{\lambda(k+1) t})+\phi_+''(t) - c \phi_+'(t) +
\phi_+(t) (1 - r - \phi_+(t)+ r\psi_+(t)),  
\end{eqnarray*}
and if $r =1$  then 
\begin{eqnarray*}
 \Gamma=
 O(\frac{(-t)^me^{\nu t}}{t-A^{-1}})+R_1(t-A^{-1})+\phi_+''(t) - c \phi_+'(t) +
\phi_+(t) ( - \phi_+(t)+ \psi_+(t)).  
\end{eqnarray*}
In each of these two cases, for some small $A_0$, we obtain $\Gamma \leq C(-t)^me^{\nu t}(\chi(\nu,c) +o(1)),$ $ A \in (0,A_0], $ $ t \to
 -\infty.$ Thus  there exists $\tau_* < i_*$ such that  $\Gamma$ is negative for all $t \leq \tau_*$ uniformly on $A \in
(0,A_0]$. On the other hand, since
$\lim_{A\to 0}\min\{\iota_1(A), \iota_2(A)\} = \min\{t_1, t_2\},$
\noindent we deduce from {\bf D2} and the above asymptotic representation of $\Gamma$  the existence of
$A_1 \in (0,A_0)$ such that $\Gamma$ is negative for all $t \in [\tau_*, i_*], \
A \in (0,A_1]$. Thus (\ref{uspdd})  holds for $t \in (-\infty,
\iota_*]$ and sufficiently small  $A \in (0,A_1]$.

\vspace{2mm}

\noindent \underline{Case II}. Suppose now that  $\iota_1(A) >  \iota_2(A)$ and let $t
\in [\iota_*, \iota^*]= [\iota_2(A),\iota_1(A)]$. We have
$$\Phi_+''(t) - c \Phi_+'(t) +
\Phi_+(t) (1 - r - \Phi_+(t)+ r\Psi_+(t)) = - r(1 - \Psi_+(t))\leq
0,$$ and, for sufficiently small $A$,
$\Psi_+''(t) - c\Psi_+'(t) +b \Phi_+(t-ch)(1-\Psi_+(t))=
$
\vspace{0mm}
$$ \left\{
\begin{array}{ll}
\Upsilon:=  \Psi_+''(t) - c\Psi_+'(t)
+b(\phi_A(t-ch)+ \phi_+(t-ch)) (1-\Psi_+(t)) <0,
    &   t \in  [\iota_*, \iota_*+ch]; \\ \psi_+''(t) - c\psi_+'(t) + b (1-\psi_+(t)) + \psi_A''(t) - c\psi_A'(t)- b\psi_A(t) <0,
    &   t \in  [\iota_*+ch, \iota^*].
\end{array}%
\right.
$$
 Here we recall that  $\psi_+''(t)
- c\psi_+'(t) +b (1-\psi_+(t))$ is negative on $[t_2+ch,t_1]$ due to assumption {\bf D4}.  On the other hand, by the same assumption, we have that,  for $ t \in  [\iota_*, \iota_*+ch]$ and all small $A$, 
\begin{eqnarray}&&\nonumber
\Upsilon= \psi_+''(t)
- c\psi_+'(t) +b\phi_+(t-ch)(1-\psi_+(t))   - b (\phi_+(t-ch)\psi_A(t) + \phi_A(t-ch)\psi_+(t))
\\  \nonumber
&& + \psi_A''(t)- c\psi_A'(t) +b\phi_A(t-ch)(1-\psi_A(t)) <0. \hspace{+0cm} 
\end{eqnarray}

\vspace{0mm}

\noindent \underline{Case III}. Similarly, if  $\iota_1(A) <
\iota_2(A)$ then  for $t \in [\iota_1(A),  \iota_2(A)]$, we obtain
\begin{eqnarray}&&\nonumber
\  \Phi_+''(t) - c \Phi_+'(t) +
\Phi_+(t) (1 - r - \Phi_+(t)+ r\Psi_+(t))  =
\Phi_+''(t) - c \Phi_+'(t) +
\Phi_+(t) (1 - \Phi_+(t))  \\ \nonumber
&=& \hspace{+0cm} 
\phi_+'' - c \phi_+' + \phi_+ (1 -  \phi_+) + \phi_A'' - c \phi_A' +
\phi_A (1 - \phi_A) -2\phi_A\phi_+ <0,
\end{eqnarray}
for all small $A$. Additionally, 
$\Psi_+''(t) - c\Psi_+'(t) +b \Phi_+(t-h)(1-\Psi_+(t))=0.
$
Since  $\Phi'_+(d_j-,A) > \Phi'_+(d_j+,A),\  \Psi'_+(d_j-,A) > \Psi'_+(d_j+,A)$ are obviously true  for all small positive $A$,  inequalities (\ref{uspdd}) are proved for small $A>0$. 

\vspace{1mm}

So let us fix some small $A^*>0$ such that $\Phi_+(t):=\Phi_+(t,A^*), $ \ $\Psi_+(t):=\Phi_+(t,A^*)$ satisfy (\ref{uspdd}). In the continuation, we  will prove the existence of  lower solutions $\Psi_-, \Phi_-: \mathbf{R} \to [0,1)$, which are defined as smooth non-decreasing functions satisfying the following system: 
\begin{equation}\label{lspdd}\left\{
\begin{array}{lll}
\Phi_-''(t) - c\Phi_-'(t) + \Phi_- (t)(1 - r- \Phi_-(t)+ r\Psi_-(t))
\geq 0, &    \\
\Psi_-''(t)- c\Psi_-'(t) +b \Phi_-(t-ch)(1-\Psi_-(t)) \geq 0, \
 & \\
 \Phi_+(t,A^*) >  \Phi_-(t) \ {\rm and \ }\Psi_+(t,A^*) >  \Psi_-(t),\
t \in {\mathbf R}.
 &
\end{array}%
\right.
\end{equation}

 We will treat separately  each of  the following cases: $r \in (0,1)$ and $r=1$.   

 \underline{Suppose that  $r\in (0,1)$.}  It follows from the definition of $\Phi_+(t,A^*),  $ $ \Psi_+(t,A^*)$  that  for some positive $k_1=k_1(A^*), \ k_2=k_2(A^*) $,
$$
\Phi_+(t,A^*) \geq k_1 e^{\lambda t}, \ \Psi_+(t,A^*) \geq k_2
e^{\lambda t}, \ t \leq 0. 
$$
Set now
$ \Psi_-(t) \equiv 0,$ and define  $\Phi_-(t), \ t \in {\mathbf R},$ as a unique (up to a
translation) traveling front solution of the KPP-Fisher equation
$$
\phi''(t) - c\phi'(t) + \phi(t) (1 - r- \phi(t)) = 0,\
\phi(-\infty)=0, \ \phi(+\infty) = 1-r >0.
$$
It is well known  \cite{GT} that $\Phi_-(t)$ is strictly
increasing and that
 $\Phi_-(t+s_0) = 0.5k_1(A^*)e^{\lambda t}+ O(e^{(\lambda+\delta) t}), $ $ t \to
-\infty,$ for some small positive $\delta$ and  for  an appropriate shift $s_0=s_0(A^*)$ 
which can be supposed to be  zero.  Hence, as a consequence of all mentioned properties of $\Phi_\pm,
\Psi_\pm, \ r \in (0,1)$, without restricting the generality, we may further assume that
the third inequality in (\ref{lspdd}) is also satisfied.

\underline{Let now $r=1$.}   
For sufficiently large $n$ (such that $\phi_{T_n}(\sigma_n) <1, \ \psi_{Q_n}(\sigma_n) <1$), we consider the following $C^1$-smooth increasing functions 
$$ \Phi_-(t,n) := \left\{
\begin{array}{ll}
    \phi_{T_n}(t+\sigma_n),
    &   t  \leq {0}, \\
   \phi_{T_n}(\sigma_n)  &   t \geq 0,
\end{array}%
\right.   \Psi_-(t,n) := \left\{
\begin{array}{ll}
    \psi_{Q_n}(t+\sigma_n),
    &   t  \leq {0}, \\
   \psi_{Q_n}(\sigma_n)  &   t \geq 0.
\end{array}%
\right.  
$$
Lemma \ref{ojala} then implies that, for all $t \in {\mathbf R}$, 
$$
\begin{array}{ll}
\Phi_-''(t,n) - c\Phi_-'(t,n) + \Phi_- (t,n)(\Psi_-(t,n)-\Phi_-(t,n))
> 0, &    \\
\Psi_-''(t,n)- c\Psi_-'(t,n) +b \Phi_-(t-ch,n)(1-\Psi_-(t,n)) >0. \
 &
\end{array}
$$
Take now $n$ sufficiently large to have $T_n <T, Q_n > Q$ and $\sigma_n < -(A^*)^{-1}$. Then 
$$
\phi_{A^*}(t)= \phi_T(t-(A^*)^{-1}) > \phi_{T_n}(t+\sigma_n), \ \psi_{A^*}(t)= \psi_Q(t-(A^*)^{-1}) > \phi_{Q_n}(t+\sigma_n),  t \leq 0. 
$$ 
As a consequence, 
\begin{eqnarray*}
\Phi_+(t,A^*)&=& \min\{1,\phi_{A^*}(t) + \phi_+(t)\} > \phi_{T_n}(t+\sigma_n)= \Phi_-(t,n), \\
\Psi_+(t,A^*)&=& \min\{1,\psi_{A^*}(t) + \psi_+(t)\} > \psi_{T_n}(t+\sigma_n)= \Psi_-(t,n), \ t \in \mathbf{R}.
\end{eqnarray*}

In order to finalize the proof of Theorem \ref{7}, for a fixed  negative number $B \leq -(1+r+b)$, we
consider  nonlinear operators
$$
      {\mathcal F}_1(\phi, \psi)(t)=   \phi(t) (1 - r- B-\phi(t)+ r\psi(t)), \quad
     {\mathcal F}_2(\phi, \psi)(t) = b \phi(t-ch)(1-\psi(t)) - B\psi(t). 
$$
It is easy to check that  ${\mathcal F}_1, {\mathcal F}_2$  are
monotone in the sense that ${\mathcal F}_j(\phi_1, \psi_1)(t) \leq
{\mathcal F}_j(\phi_2, \psi_2)(t),$ $  \ t \in {\mathbf R},$ if $0\leq
\phi_1(t) \leq \phi_2(t)\leq 1, \  0\leq \psi_1(t) \leq
\psi_2(t)\leq 1, \ t \in {\mathbf R} $.
 Let $z_1 <0< z_2$ be the real roots
of the equation $z^2-cz+B=0$.
Then  every bounded solution $(\phi, \psi)$ of differential equations in  (\ref{3ade}) should satisfy the system of integral equations
\begin{equation}\label{pfoin}
\phi(t)={\mathcal N}_1(\phi, \psi)(t), \  \psi(t)= {\mathcal
N}_2(\phi, \psi)(t), \quad {\rm where}
\end{equation}
$$
{\mathcal N}_j(\phi, \psi)(t):=\frac{1}{z_2 - z_1}
\left(\int^t_{-\infty} e^{z_1 (t-s)}{\mathcal F}_j(\phi, \psi)(s)ds + \int_t^{+\infty}e^{z_2
(t-s)}{\mathcal F}_j(\phi, \psi)(s)ds\right). 
$$
Conversely, each positive strictly monotone bounded solution $(\phi,
\psi)$ of  (\ref{pfoin}) yields a wavefront for 
(\ref{3ade}). It is clear that the operators ${\mathcal N}_j$  are also monotone.  Additionally, it is easy to see that ${\mathcal N}_j(\phi, \psi)(t)$ is increasing if both $\phi, \psi : {\bf R} \to [0,1]$ are increasing functions.

Hence, taking into account (\ref{uspdd}), (\ref{lspdd})  and Lemma
\ref{imp}, we conclude that  
\begin{eqnarray*}
\Phi_-(t) &\leq & \Phi_-^{(1)}(t):= {\mathcal N}_1(\Phi_-,
\Psi_-)(t) \leq  {\mathcal N}_1(\Phi_+,
\Psi_+)(t):= \Phi_+^{(1)}(t)  \leq  \Phi_+(t), 
 \\
\Psi_-(t)  &< &\Psi_-^{(1)}(t):= {\mathcal N}_2(\Phi_-,
\Psi_-)(t)  \leq {\mathcal N}_2(\Phi_+,
\Psi_+)(t):= \Psi_+^{(1)}(t) \leq \Psi_+(t).
\end{eqnarray*}
Therefore the sequences of positive uniformly bounded (by $0$ from below  and by $1$ from above) monotone continuous functions 
\begin{equation} \label{apli} \hspace{7mm}
 \Psi_-^{(n+1)}(t) =  {\mathcal N}_2(\Phi_-^{(n)}, \Psi_-^{(n)})(t), \ \Phi_-^{(n+1)}(t) =  {\mathcal N}_1(\Phi_-^{(n)}, \Psi_-^{(n)})(t), \  n =1,2, \dots,
\end{equation}

\vspace{-7mm} 

$$
\hspace{-7mm} {\rm and} \ \   \Psi_+^{(n+1)}(t) =  {\mathcal N}_2(\Phi_+^{(n)}, \Psi_+^{(n)})(t), \ \Phi_+^{(n+1)}(t) =  {\mathcal N}_1(\Phi_+^{(n)}, \Psi_+^{(n)})(t), \  n =1,2, \dots,
$$
are strictly increasing and decreasing, respectively.  Set $\Phi = \lim \Phi_-^{(n)}, \
\Psi = \lim \Psi_-^{(n)}$, then 
\begin{equation} \label{zer}
\Phi_- \leq  \Phi \leq \Phi_+, \ \Psi_-\leq \Psi \leq  \Psi_+.
\end{equation}
Furthermore, a direct application of the Lebesgue's dominated convergence theorem to (\ref{apli})
shows that the pair $(\Phi, \Psi)$  solves system (\ref{pfoin}). Since $\Phi(t) >0$ for all $t$, we may conclude from  (\ref{pfoin}) that $\Psi(t) >0,$ $t \in {\mathbf R}$.  Note also that $\Phi(-\infty) = \Psi(-\infty)=0$   in virtue of (\ref{zer}).  Now, since $\Phi, \Psi$ are positive, increasing and bounded functions,  the values of $\Phi(+\infty), \Psi(+\infty)$ are finite and positive.  A standard argument based on the Barbalat lemma (cf. \cite{wz}) shows that $\Phi(+\infty) = \Psi(+\infty)=1$.   

Finally, the  validity of asymptotic formula (\ref{a1}) follows from (\ref{zer}).  To prove (\ref{a2}), we first observe that, due to  (\ref{zer}) and Theorem \ref{mda}, there exists $T_b <0$ such that $\Phi(t)$ and $\Psi(t)$ are bounded (from below and from above) by $c_ie^{\lambda t}$ for some $c_i >0$ and all $t\leq T_b$. Then we can apply Proposition 7.2 from \cite{MP} to the first equation of (\ref{3ade}) in order to obtain the desired formula for $\Phi(t)$.  Using this formula and the change of variables $y =\psi - be^{\lambda(t-ch)}/(1-r)$,  we then 
get easily the second formula of (\ref{a2}), cf. the proof of Lemma \ref{mainas} and Corollary \ref{posl}.   \hfill $\square$
\end{proof}

\subsection{Proof of Theorem \ref{main1} } \label{4.2}
The simplest form of  regular super-solutions $\phi_+,\psi_+$ is exponential,
we can write them as
\begin{equation} \label{fp}
\phi_+(t)= e^{\nu t}, \ \psi_+(t)=De^{\nu t} , \  D =be^{-\nu
ch}/(c\nu-\nu^2),
\end{equation}
where $D$ is chosen in such a way that the second inequality
 in {\bf D2} as well as {\bf D4} were satisfied
(details are given below).  In order to simplify the notation,
in the sequel we will write $b'=be^{-\nu ch}$.  
\vspace{-5mm}
\begin{lemma} \label{lii} Suppose that $\nu \not= j\lambda$ is
close to $c/2, r >0,$ and
\begin{equation}
\label{arti}
(c\nu - \nu^2)(1+ 1/b') > 1, \ c\nu - \nu^2 \not= b'.
\end{equation}
Then (\ref{fp})  determines a regular super-solution for (\ref{3ade}).
\end{lemma}

\vspace{-4mm}

\begin{proof} Clearly,  {\bf D1} is satisfied with $C_1=1,$ 
$
t_1 = \nu^{-1}\ln[(c\nu-\nu^2)/b'] \not=0, \ t_2=0.
$
Still we have to check hypotheses {\bf D2, D3, D4}. Depending on the
sign of $t_1$, we will analyze the next two cases:

\underline{Case I.}\quad $t_1 >0=t_2$ or, equivalently,  $ 0<
b'/(c\nu-\nu^2) < 1$. If $t \in [0,t_1]$, then {\bf D4} holds
because
$$
\psi_+''(t) - c\psi_+'(t) +b \min\{1, \phi_+(t-ch)\}(1-\psi_+(t)) = \left\{
\begin{array}{ll}
    -b'e^{\nu t}\psi_+(t) <0,
    &   t \in  [0, ch]; \\  b\left(1- (1+ \frac{b}{c\nu-\nu^2})e^{\nu (t-ch)}\right)<0,
    &   t \in  [ch, t_1].
\end{array}%
\right.
$$
If $t\leq 0$, we have that
\begin{eqnarray} \label{vzhe}
\hspace{-7mm} &&\psi_+''(t) - c\psi_+'(t) +b \phi_+(t-ch)(1-\psi_+(t))=
-b'De^{2\nu t}<0; \nonumber \\
\hspace{-7mm}  && \phi_+''(t) - c\phi_+'(t) + \phi_+(t) (1 - r - \phi_+(t)+
r\psi_+(t))  = e^{\nu t}\{\chi(\nu,c)- \phi_+(t)+ r\psi_+(t)\} < 0. 
\end{eqnarray}

\underline{Case II.}  Next, let $t_1  < t_2=0$ so that
$b'/(c\nu-\nu^2) > 1$. If $t \in[t_1,0]$  then
\begin{eqnarray*}
&& \phi_+''(t) - c\phi_+'(t) +
\phi_+(t) (1 - \phi_+(t)) =
e^{\nu t}(1-e^{\nu t}+
\nu^2-c\nu)  \\
&\leq &
e^{\nu t}(1-e^{\nu t_1}+ \nu^2-c\nu)= e^{\nu t}(1-(c\nu- \nu^2)(1+
1/b')) <0,
\end{eqnarray*}
and {\bf D3} holds. Now, for $t \leq t_1$, condition {\bf D2} is
true since
\begin{eqnarray} \label{obs} 
\phi_+''(t) - c \phi_+'(t) + \phi_+(t) (1 - r - \phi_+(t)+
r\psi_+(t)) &=&\\ \nonumber 
e^{\nu t}(\chi(\nu,c)+e^{\nu t}(- 1+ rb'/(c\nu- \nu^2)) \leq e^{\nu
t}\max\{\chi(\nu,c), 1 - (c\nu - \nu^2)(1+ 1/b')\}&<&0;\\ \nonumber 
\psi_+''(t) - c\psi_+'(t) +b \phi_+(t-ch)(1-\psi_+(t))=
-b'De^{2\nu t}&<&0. 
\end{eqnarray}
This completes the proof of the lemma.  \hfill $\square$
\end{proof}
\begin{corollary}\label{difi}
The existence statement of Theorem \ref{main1} holds true.
\end{corollary}
\begin{proof}
First, we assume that $c > c_\#$.  Then clearly there is a positive $\nu$ meeting all requirements 
of Lemma \ref{lii}. This assures the existence  of a regular super-solution for  (\ref{3ade}).  By Theorem 
\ref{7},  system (\ref{3ade}) has a positive monotone wavefront.  

Next, we consider the case when $c = c_\#, \ r \in (0,1]$. Let $c_j \downarrow c_\#$ be a strictly decreasing 
sequence of velocities and $(\phi_j,\psi_j)$ be a sequence of corresponding traveling fronts (existing in virtue of the 
first part of the proof).  Since 
$$
0 = \phi_j(-\infty)+\psi_j(-\infty) < \phi_j(t)+\psi_j(t) < \phi_j(+\infty)+\psi_j(+\infty) =2
$$ 
and the function $\phi_j(t)+\psi_j(t)$ is increasing in $t$ for each fixed $j$, we may assume that 
$
\phi_j(0)+\psi_j(0) = 3/2, \quad j =1,2,3, \dots
$
Using the standard compactness arguments and then applying the Lebesgue's dominated convergence theorem to the system of integral equations (\ref{pfoin}): 
$$
\phi_j(t)={\mathcal N}_1(\phi_j, \psi_j,c_j)(t), \  \psi_j(t)= {\mathcal
N}_2(\phi_j, \psi_j,c_j)(t), 
$$
we may assume, without restricting the generality,  
that $\lim_j(\phi_j,\psi_j) = (\hat\phi, \hat\psi)$ uniformly on bounded intervals, where $(\hat\phi, \hat\psi)$
is a monotone solution of (\ref{3ade}) with $c=c_\#$.  Since $(\hat\phi, \hat\psi)(\pm\infty)$ are steady state solutions of 
(\ref{3ade}) and 
$\hat\phi(-\infty)+  \hat\psi(-\infty) \leq \hat\phi(0)+  \hat\psi(0) =3/2 \leq \hat\phi(+\infty)+  \hat\psi(+\infty),$
we find that necessarily 
$$
\hat\phi(-\infty)=0, \ \hat\psi(-\infty) \in [0,1] , \ \hat\phi(+\infty) =  \hat\psi(+\infty) =1,
$$
(if  $\hat\phi(-\infty)>0$, then $\hat\psi(-\infty) = \hat\phi(-\infty) =1$ and thus
 $\hat\phi(0)+  \hat\psi(0) =2$, a contradiction).  To finish the proof of the corollary, we have to 
 establish that $\hat\psi(-\infty)=0$.  In order to prove this,  we can apply the part [A] (for $r \in (0,1)$) and the part [C] (when $r =1$)  of Theorem \ref{mda}
 to find that  either $\psi_j(t) < K \phi_j(t), \ t \in \mathbf{R}$ (for $r \in (0,1)$) or 
 $\psi_j(t) < \sqrt{M \phi_j(t)}, \ t \in \mathbf{R}$ (for $r =1$). Therefore either 
 $\hat \psi(t) \leq K \hat \phi(t), \ t \in \mathbf{R},$ 
 or  $\hat \psi(t) \leq \sqrt{M \hat \phi(t)}, \ t \in \mathbf{R},$
 so that $\hat\psi(-\infty)=0$.   \hfill $\square$
 \end{proof}
\subsection{Proof of Theorem \ref{main2}} \label{4.3a}
Let now $c < c_\#$ so that for some $\nu$ close to $c/2$ 
\begin{equation}\label{oz}
0<(c\nu - \nu^2)(1+ \frac{1}{b'}) \leq  1, \ c\nu - \nu^2 \not= b', \ \nu \not= j\lambda.
\end{equation}
 Analyzing the proof
of Lemma \ref{lii} under these assumptions,  we see that  $t_1<0=t_2$  and the main obstacle to develop successfully  the proof of Case II appears when we want to estimate expression  (\ref{obs}) near $t_1<0$.
Therefore we may expect that, after an appropriate modification
of super-solutions  (\ref{fp}) in some neighborhood of $t_1$,
the result of Theorem \ref{main1}  can be  improved.
Below, we develop this idea  by considering 
$\phi_+(t)=e^{\nu t}, \  t \in {\mathbf R}$, and $C^1-$ smooth function
\begin{equation}\label{pne}
\psi_+(t)= \left\{
\begin{array}{ll}
     De^{\nu t},
    &  {\rm if} \ t \leq t_*;  \\
      p +qt , & {\rm if}\  t > t_*.
\end{array}%
\right.
\end{equation}
Here $p, q, t_*$ will be chosen to satisfy the first inequality in {\bf D2} for all $t \in {\mathbf R}$.
The mentioned inequality can be written as
\begin{equation}\label{pp}
\psi_+ (t)< \gamma(t):=r^{-1}(c\nu-\nu^2+r-1 +e^{\nu t}).
\end{equation}
Now, assuming  (\ref{oz}) and analyzing the mutual positions
of convex graphs of the functions $\gamma(t)$ and $De^{\nu t}$, we  deduce that
these graphs should have exactly one point of intersection (or tangency) below the
level $y =1$.  Indeed,
otherwise $c\nu-\nu^2+r-1 >0$ implies that  $\gamma(t) > De^{\nu t}$ for all
$t $ where $\gamma(t) \leq 1$. As a consequence, $1=\gamma(s_0) > De^{\nu s_0}$  at some $s_0$ which 
implies (\ref{arti}), a contradiction.  
 
The above consideration and a direct computation show that there
exists a a unique line $y=p+qt$ which is tangent to the graphs of
 $De^{\nu t}$  and $\gamma(t) $ at the respective points $t_* < t^*$.
From the tangency conditions
$
q= D\nu e^{\nu t_*} = \gamma'(t^*), \  p = De^{\nu t_*} - qt_* = \gamma(t^*)-qt^*,
$
it follows easily  that
$$
p = \frac{q}{\nu} \ln \frac{D\nu e}{q}, \  q = \frac{(c\nu-\nu^2+r-1)\nu}{r\ln (rD) }, \  t_* = \frac{1}{\nu}\ln\frac{q}{D\nu}.
$$
It follows from the above construction that $\psi_+$ defined by   (\ref{pne}) is $C^1$- smooth
and $\psi_+(t) <  \gamma(t)$ for all $t\not= t^*$.  It is clear that, after making an arbitrarily small
change of  $p,q,t_*$,  we may assume that   $\psi_+(t) <  \gamma(t)$ for all $t \in {\mathbf R}$.

Hence, taking $\psi_+$ as in (\ref{pne}) and $\phi_+(t)=e^{\nu t}$,
we have to check only the second inequality  in {\bf
D2} on the interval $[t_*, +\infty)$.  This inequality can be
written as
$
b'(1-p)/q < b't +ce^{-\nu t}.
$
Since $y = b't +ce^{-\nu t}$ has a unique critical point (an absolute
minimum) at $t'= \nu^{-1}\ln(c\nu/b')$, the latter inequality 
amounts  to
\begin{equation}\label{pnu}
(1-p)\nu < q \ln (ec\nu/b').
\end{equation}
After recalling the definition of $p$ and $D$  and taking into account
that $\nu$ can be chosen as  close to $c/2$ as we want,
we rewrite  (\ref{pnu}) as $\omega < 2(2 + \ln \omega),$
where
$$
\omega = \frac{2r}{c^2/4 +r-1}\ln \frac{4b'r}{c^2}\quad  (=\frac cq ).
$$
Notice here that  the assumptions $c^2 > 4(1-r)$ and $c^2
<4/(1+(b')^{-1})$ imply   $r(b'+1)>1$ and  $c^2 < 4b'r$ so that $\omega
>0$.  Furthermore, since $\omega$ is decreasing in $c^2/4$, we find that
$$
\omega > \frac{2r}{b'/(1+b') +r-1}\ln r(b'+1)= \frac{2r(1+b')}{r(1+b')
-1}\ln r(b'+1)> 2.
$$
A direct graphical analysis shows that the interval $\omega \in
(0.14555\dots, 8.21093\dots)$ gives the solution of $\omega < 2(2 + \ln \omega)$. In
consequence, since we additionally have $\omega >2$, the latter
inequality is equivalent to $\omega < \omega_* =8.21
\dots$ which can be written as (\ref{posmotrim}). 
This proves Theorem \ref{main2}.  \hfill $\square$

\section{Proof of Theorem  \ref{mumi>2}} \label{5th}
\noindent The proof is divided into three claims. 

\noindent \underline{Claim I}:  {\it The propagation speed $c_\star$ is unique. } 
Indeed, suppose that $(\phi_1,\psi_1,c_1)$ and $(\phi_2,\psi_2,c_2)$, $c_1<c_2$, solves the nonlinear eigenvalue problem (\ref{3ade}). It follows from 
Lemmas  \ref{pred>1} and \ref{mainas}  that there exist $p <q$ such that  $\psi_1(t) > \psi_2(t)$ for all $ t \in \R\setminus[p,q]$. As a consequence, the closed set 
$$
\mathcal{S}:= \{s : \psi_1(t+s) \geq \psi_2(t), \ t \in \R\} \not= \R
$$
is non-empty and has a finite $s_*:=\inf \mathcal{S}$.  It is clear that 
$
\psi_1(t+s_*) \geq \psi_2(t), \ t \in \R, 
$ 
and since always  $\psi_1(t+s_*) > \psi_2(t) $ for $t \ll -1 $ and $t \gg 1$,
we deduce that $
\psi_1(\tau+s_*) = \psi_2(\tau) 
$ 
at some point $\tau$ (otherwise $s_\star > \inf \mathcal{S}$). By a similar argument,  there exists 
$t_*$ such that 
$
\phi_1(t+t_*) \geq \phi_2(t), \ t \in \R, 
$ 
and $
\phi_1(T+t_*) = \phi_2(T) 
$ 
for  some $T$. 
Suppose first that $t_* \leq s_*$. 
Without restricting the generality, we may assume 
that $s_* =0, \tau =0$. Then $t_* \leq 0$ so that $\phi_1(t) \geq \phi_2(t), \ t \in \R,$ and  thus we get
$$
0=(\psi_1-\psi_2)''(0)- c_1(\psi_1-\psi_2)'(0)+ (c_2-c_1)\psi_2'(0)+ b(\phi_1(-c_1h)-\phi_2(-c_2h))(1-\psi_1(0))>0,
$$
a contradiction.  Next, suppose  that $t_* > s_*$. 
We may assume again that  that $t_* =0, T =0$. Then $s_* <0$ so that $\psi_1(t) > \psi_2(t), \ t \in \R,$ and  thus we get
$$
0=(\phi_1-\phi_2)''(0)- c_1(\phi_1-\phi_2)'(0)+ (c_2-c_1)\phi_2'(0)+ r\phi_1(0)(\psi_1(0)- \psi_2(0))>0,
$$
a contradiction. Hence $c_1=c_2$ and Claim I is proved.  

\noindent \underline{Claim II}:  {\it  $c_\star \leq c_m:= \min\{c_\#, c_\circ\}$. } Let $(\phi_*,\psi_*,c_\star)$ be the solution of  (\ref{3ade}). On the contrary, suppose that $c_\star > c_m$ and take an arbitrary $c' \in (c_m, c_\star)$.  Then  $(\phi_*,\psi_*,c')$ is a lower solution:
$$
\phi''_*(t) -c'\phi_*'(t) + \phi_*(t)(1-r-\phi_*(t)+r\psi_*(t)) >0, \quad \psi_*''(t)-c'\psi'_*(t)+b\phi_*(t-c'h)(1-\psi_*(t))>0, \ t \in \R.  
$$
For the same $c'$ we consider  the upper solutions $
\Phi_+(t)= \min\{1,\phi_+(t)\}, \ \Psi_+(t)=\min\{
1,\psi_+(t)\}, 
$
 with $\phi_+, \psi_+$ defined in Subsections \ref{4.2}, \ref{4.3a}. By Lemma \ref{pred>1}, we may suppose (possibly, after a translation of $(\phi_*,\psi_*)$) that 
$\phi_*(t)  < \Phi_+(t), \ \psi_*(t) < \Psi_+(t), \ t \in \R$.  But then there exists (cf. the last part of Subsection \ref{upp}, starting  from formula (\ref{pfoin})) a monotone traveling front propagating at the velocity $c' < c_\star$. However, this contradicts to Claim I. 

\noindent \underline{Claim III}:  {\it Set $c_\star(h): = c_\star(r,b,h)$ for some fixed $r,b >0$. Then $c_\star(h)$ is a non-increasing function on its domain. } Suppose that $c_\star(h_1) > c_\star(h_2)$ for some $h_1 > h_2$.   Let $(\phi_j,\psi_j,c_\star(h_j))$ be respective solutions of  (\ref{3ade}). Then, for a fixed $c \in ( c_\star(h_2), c_\star(h_1))$, it holds
$$
\phi''_1(t) -c\phi_1'(t) + \phi_1(t)(1-r-\phi_1(t)+r\psi_1(t)) >0, \quad \psi_1''(t)-c\psi'_1(t)+b\phi_1(t-ch_1)(1-\psi_1(t))>0, \ t \in \R,  
$$
$$
\phi''_2(t) -c\phi_2'(t) + \phi_2(t)(1-r-\phi_2(t)+r\psi_2(t)) <0, \quad \psi_2''(t)-c\psi'_2(t)+b\phi_2(t-ch_1)(1-\psi_2(t))<0, \ t \in \R. 
$$
Moreover, due to Lemmas  \ref{pred>1} and \ref{mainas},  we may assume that 
$
\phi_1(t) < \phi_2(t), \ \psi_1(t) < \psi_2(t), \ t \in \R. 
$
Therefore $(\phi_j,\psi_j, c), \ j=1,2,$ forms a pair of upper and lower solutions for (\ref{3ade}) considered 
with $c$ and $h_1$. As a consequence, system (\ref{3ade})  with $h=h_1$ has two different propagation speeds: $c$ and $c_\star(h_1)> c$.  However, this  is a contradiction with  Claim I.   \hfill $\square$

\section*{Acknowledgments}  The authors express their gratitude  to the  referee, whose critical comments and valuable suggestions helped to 
improve the original version of this paper. This research was supported by FONDECYT (Chile), projects  1080034 and 1110309, and  by CONICYT (Chile) 
through PBCT program ACT-56.

\end{document}